\documentclass{amsart}
\usepackage[a4paper,margin=1in,includeheadfoot]{geometry}

\usepackage{natbib}
\usepackage{parskip}
\usepackage{amssymb,amsmath}
\usepackage[matrix,arrow,arc]{xy}
\usepackage{array}
\usepackage{float}\floatstyle{boxed}\restylefloat{figure}
\usepackage[colorlinks=true]{hyperref}
\usepackage[noend]{algpseudocode}
\usepackage{algorithm}

\newcommand\BBC{\mathbb C}
\newcommand{\cyclic}[1]{Z_{#1}}
\newcommand\Ind{\operatorname{Ind}}

\newcommand{\symmetric}[1]{S_{#1}}

\newcommand{\GAP}{\textsf{GAP}}
\newcommand{\CHEVIE}{\textsf{CHEVIE}}
\newcommand{\ZigZag}{\textsf{ZigZag}}

\numberwithin{equation}{section}

\theoremstyle{plain}
\newtheorem{theorem}{Theorem}
\numberwithin{theorem}{section}
\newtheorem*{conjecture1}{Conjecture 1}

\newtheorem*{conjectureb}{Conjecture B}
\newtheorem*{conjecturec}{Conjecture C}
\theoremstyle{definition}
\newtheorem{remark}[theorem]{Remark}

\title[Computations for rank three and four Coxeter groups]
{Computations for Coxeter arrangements and Solomon's descent algebra:\\
 Groups of rank three and four}

\author[M.~Bishop]{Marcus Bishop} \address{M.B.: Fakult\"at f\"ur Mathematik\\
  Ruhr-Universit\"at Bochum\\ D-44780 Bochum, Germany}
\email{marcus.bishop@rub.de}

\author[J.~M.~Douglass]{J. Matthew Douglass} \address{J.M.D.: Department of
  Mathematics\\ University of North Texas\\ Denton TX, USA 76203}
\email{douglass@unt.edu}

\author[G.~Pfeiffer]{G\"otz Pfeiffer} \address{G.P.: School of Mathematics,
  Statistics and Applied Mathematics\\ National University of Ireland,
  Galway\\ University Road, Galway, Ireland}
\email{goetz.pfeiffer@nuigalway.ie}

\author[G.~R\"ohrle]{Gerhard R\"ohrle} \address{G.R.: Fakult\"at f\"ur
  Mathematik\\ Ruhr-Universit\"at Bochum\\ D-44780 Bochum, Germany}
\email{gerhard.roehrle@rub.de}

\begin{document}

\maketitle
\allowdisplaybreaks

\begin{abstract}
  In recent papers we have refined a conjecture of Lehrer and Solomon
  expressing the characters of a finite Coxeter group $W$ afforded by the
  homogeneous components of its Orlik-Solomon algebra as sums of characters
  induced from linear characters of centralizers of elements of $W$. Our
  refined conjecture also relates the Orlik-Solomon characters above to the
  terms of a decomposition of the regular character of $W$ related to the
  descent algebra of~$W$. A consequence of our conjecture is that both the
  regular character of $W$ and the character of the Orlik-Solomon algebra
  have parallel, graded decompositions as sums of characters induced from
  linear characters of centralizers of elements of $W$, one for each
  conjugacy class of elements of $W$. The refined conjecture has been proved
  for symmetric and dihedral groups. In this paper we develop algorithmic
  tools to prove the conjecture computationally for a given finite Coxeter
  group. We use these tools to verify the conjecture for all finite Coxeter
  groups of rank three and four, thus providing previously unknown
  decompositions of the regular characters and the Orlik-Solomon characters
  of these groups.
\end{abstract}


\section{Introduction}\label{IntroductionSection}

Let $W$ be a finite Coxeter group of rank $n$ and let $V$ be a finite
dimensional, complex vector space affording a faithful representation of $W$
such that each element in a Coxeter generating set $S$ of $W$ acts on $V$ as
a reflection. Let $M$ be the complement in $V$ of the union of the
fixed-point hyperplanes of the reflections in $W$. Then $M$ is a $W$-stable,
open subset of $V$ and the action of $W$ on $M$ determines a representation
of $W$ on $H^p(M)$, the $p^{\text{th}}$ singular cohomology group of
$M$. Let $\omega_W^p$ denote the character
of 
this representation and let $\omega_W= \sum_{p\geq 0} \omega_W^p$ denote the
character of the representation of $W$ on the cohomology ring $H^\bullet(M)=
\bigoplus_{p\ge 0} H^p(M)$.  The character $\omega_W$ has been computed by
Lehrer and others, see \cite{blairlehrer:cohomology} and
\cite{felderveselov:coxeter}.

\cite{lehrersolomon:symmetric} conjectured that $\omega_W^p$ is a sum of
characters induced from linear characters of centralizers of elements of
$W$.  Conjecture~1 below is a more precise version of the Lehrer-Solomon
conjecture that in addition to describing $\omega_W^p$ as a sum of induced
characters, also describes the regular character $\rho_W$ of $W$ in a
similar way.

\begin{conjecture1}[\cite{douglasspfeifferroehrle:coxeter},
  Conjecture~2.1]\label{conj1}

  Suppose that $W$ is a finite Coxeter group and that $\mathcal{R}$ is a set
  of conjugacy class representatives of $W$. Then for each $w\in\mathcal{R}$
  there exists a linear character $\varphi_w$ of $C_W(w)$ such
  that
  \[
  \rho_W= \sum_{w\in\mathcal{R}} \Ind_{C_W(w)}^W \varphi_w
  \quad\text{and}\quad\omega_W= \epsilon \sum_{w\in\mathcal{R}}
  \Ind_{C_W(w)}^W (\alpha_w \varphi_w )
  \]
  where $\epsilon$ is the sign character of $W$ and $\alpha_w$ is the
  composition of $\det$ with restriction to the $1$-eigenspace of $w$.
  Moreover, if $\mathcal R_p$ is the set of $w$ in $\mathcal R$ such that
  the codimension in $V$ of the $1$-eigenspace of $w$ is $p$, then
  \[
  \omega_W^p= \epsilon \sum_{w\in\mathcal{R}_p} \Ind_{C_W(w)}^W (\alpha_w
  \varphi_w ).
  \]
\end{conjecture1}
The main result in \cite{douglasspfeifferroehrle:coxeter} is a proof of
Conjecture~1 for symmetric groups. In
\cite{douglasspfeifferroehrle:inductive} we developed an inductive approach
that would lead to a proof of Conjecture~1. As an application of the method,
we proved Conjecture~1 for dihedral groups. In this paper we implement the
inductive procedure and give a computational proof that Conjecture~1 holds
for Coxeter groups with ranks three and four. The experimental evidence
provided by the computations in this paper suggests that cohomology encodes
as yet undiscovered symmetries of the complex manifolds that arise as
complements of Coxeter arrangements.

Explicitly computing characters $\varphi_{w}$ that satisfy the conclusion
of Conjecture~1 can in principle be accomplished by testing whether
\[
\rho_W=\sum_{w\in\mathcal{R}} \Ind_{C_W(w)}^W\varphi_w \quad \text{and}\quad
\omega_W^p= \epsilon \sum_{w\in\mathcal{R}_p} \Ind_{C_W(w)}^W (\alpha_w
  \varphi_w ) \quad (0\leq p\leq n)
\]
for each tuple $(\varphi_w)_{w\in\mathcal{R}}$ of linear characters of
$C_W(w)$ for $w$ in $\mathcal R$.  The inductive approach parses
Conjecture~1 into distinct, more manageable components known as
Conjectures~B and~C, which we now describe.  For a subset $L$ of $S$ denote
the parabolic subgroup $\langle L \rangle$ of $W$ by $W_L$. A conjugacy
class $\mathcal C$ in $W$ is called \emph{cuspidal} if $\mathcal C\cap
W_L=\emptyset$ for every proper subset $L$ of $S$.  Then Conjecture~B
describes the character $\omega_W^n$ of $W$ as a sum of characters induced
from centralizers of cuspidal conjugacy classes. Note that $\omega_W^n$ is
the highest degree nonzero component of $\omega_W$, because $H^n(M)$ is the
highest degree non-vanishing cohomology group of~$M$.  Furthermore,
Conjecture~B also relates $\omega_W^n$ to an appropriate term $\rho_W^n$ of
a decomposition of the regular character of~$W$ arising from a complete
set of primitive orthogonal idempotents of the descent algebra of $W$.

\begin{conjectureb}[\cite{douglasspfeifferroehrle:inductive}, Conjecture~B]
  Suppose that $W$ is a finite Coxeter group and that
  $\mathcal{C}$ is a set of representatives of the cuspidal conjugacy
  classes of $W$. Then for each $w\in\mathcal{C}$ there exists a linear
  character $\varphi_w$ of $C_W(w)$ such that
  \[
  \rho_W^n=\sum_{w\in\mathcal{C}} \Ind_{C_W(w)}^W \varphi_w =\epsilon
  \omega_W^n.
  \]
\end{conjectureb}

Conjecture~C is a relative version of Conjecture~B for the pair $(W,W_L)$,
where $W_L$ is a fixed parabolic subgroup of rank $r$ and the overgroup $W$
varies.  It mirrors Conjecture~B for the group $W_L$, but also encodes
information about the embedding of $W_L$ in its normalizer. The characters
$\rho^r_{W_L}$ and $\omega^r_{W_L}$ are replaced by extensions
$\widetilde{\rho^r_L}$ and $\widetilde{\omega^r_L}$ to $N_W(W_L)$. The
extensions $\widetilde{\rho^r_L}$ and $\widetilde{\omega^r_L}$ are discussed
in \S\ref{AlgebrasSection}.

\begin{conjecturec}[\cite{douglasspfeifferroehrle:inductive}, Conjecture~C]
  Suppose that $W$ is a finite Coxeter group, that $W_L$ is a proper,
  parabolic subgroup of $W$ with rank $r$, and that $\mathcal{C}_L$ is a set
  of representatives of the cuspidal conjugacy classes of $W_L$. Then for
  each $w\in\mathcal{C}_L$ the linear character $\varphi_w$ of $C_{W_L}(w)$
  in Conjecture~B extends to a linear character $\widetilde{\varphi_w}$ of
  $C_W(w)$ such that
  \[
  \widetilde{\rho_L^r}=\sum_{w\in\mathcal{C}_L} \Ind_{C_W(w)}^{N_W(W_L)}
  \widetilde{\varphi_w} =\epsilon\alpha_L \widetilde{\omega_L^r},
  \]
  where $\alpha_L$ is the composition of $\det$ with restriction to the
  subspace of fixed points of $W_L$.
\end{conjecturec}

It is shown in \cite{douglasspfeifferroehrle:inductive} that if Conjecture~B
holds for $W$ and Conjecture~C holds for all pairs $\left(W,W_L\right)$
where $W_L$ is a proper, parabolic subgroup, then Conjecture~1 holds for
$W$.

{}From a computational standpoint, proving Conjecture~B for a given group $W$
requires solving three separate problems.
\begin{enumerate}
\item How can the character $\rho_W^n$ be computed?
\item How can the character $\omega_W^n$ be computed?
\item How can linear characters $\varphi_{w}$ for $w$ in $\mathcal{C}$ be
  found so that the first equality in Conjecture~B holds?
\end{enumerate}
In \S\ref{ImplementationSection} we present methods to solve these problems
using the \GAP\ programming system \cite{gap3} with the \CHEVIE\
\cite{chevie} and \ZigZag\ \cite{zigzag} packages. This includes new
algorithms for computing with non-broken circuit bases of Orlik-Solomon
algebras of arbitrary hyperplane arrangements. These algorithms have been
added to the \ZigZag\ package developed and maintained by the second author.

It turns out that for groups of rank at most four, Conjecture~C may be
proved directly, as discussed in \S\ref{sec:prelim}. Thus, as a
consequence of our computations, we can state the following theorem.

\begin{theorem}\label{thm:all}
  Conjecture~1 holds for all finite Coxeter groups of rank at most four.
\end{theorem}

Some of our current methods are sufficient to treat Conjecture~B for
somewhat larger groups, but are computationally too expensive to be able to
handle the largest exceptional Coxeter groups. In addition, as the rank
increases, the computations needed to verify Conjecture~C become more
complicated. In future work we hope to develop additional computational
techniques to be able to efficiently verify the conjectures for groups of
rank up to eight, thus verifying the conjectures for all exceptional Coxeter
groups.

The rest of this paper is organized as follows. In \S\ref{sec:prelim} we
review the constructions from \cite{douglasspfeifferroehrle:coxeter} and
\cite{douglasspfeifferroehrle:inductive} and show how our computations lead
to a proof of Theorem~\ref{thm:all}. In \S\ref{ImplementationSection} we
describe the algorithms to solve the three computational problems listed
above and their implementation in \GAP. In \S\ref{sec:results} we present
the results of our computations for Conjecture~B for rank three and four
Coxeter groups. We include somewhat more information about centralizers and
the linear characters $\varphi_w$ than is necessary to prove
Conjecture~1. This extra information can be used to test hypotheses about
how the characters $\varphi_w$ might be chosen for an arbitrary Coxeter
group. Finally, Conjecture~C is known to hold if $W_L$ is a bulky parabolic
subgroup of $W$ for which Conjecture~B holds. In the appendix we list the
all the bulky parabolic subgroups of finite, irreducible, Coxeter groups.

\section{Preliminaries and Main Results}\label{sec:prelim}

\subsection{Coxeter Groups and the Orlik-Solomon Algebra} 
\label{AlgebrasSection}

In this subsection we briefly review the constructions in
\cite{douglasspfeifferroehrle:coxeter} and
\cite{douglasspfeifferroehrle:inductive}. Recall that an element of $W$ is
called \emph{cuspidal} if none of its conjugates lies in a proper parabolic
subgroup of $W$. A conjugacy class is called \emph{cuspidal} if its elements
are all cuspidal. It follows from the fact that the proper parabolic
subgroups of $W$ arise as pointwise stabilizers of proper subspaces of $V$
that an element is cuspidal if and only if its $1$-eigenspace has
codimension $|S|$ in $V$.  It is shown in \cite{geckpfeiffer:characters}
that up to the natural action of $W$, the conjugacy classes in $W$ are
parameterized by pairs $(W_1, \mathcal C_1),$ where $W_1$ is a parabolic
subgroup of $W$ and $\mathcal C_1$ is a cuspidal conjugacy class in $W_1$.

Let $T=\{\,w^{-1}sw\mid s\in S,\, w\in W\,\}$ be the set of reflections in
$W$. For $t$ in $T$ let $H_t$ be the hyperplane in $V$ fixed by $t$. Let $E$
be a $\BBC$-vector space with basis $\{\,e_t\mid t\in T\,\}$. The
\emph{Orlik-Solomon algebra} $A(W)$ is the quotient of the exterior algebra
of $E$ by the ideal generated by elements of the form
\begin{equation}\label{ARelation}
  \sum_{i=1}^m(-1)^ie_{t_1}e_{t_2}\cdots\widehat{e_{t_i}} \cdots e_{t_m} 
\end{equation}
for every set $\{H_{t_1},H_{t_2},\dots,H_{t_m}\}$ of linearly dependent
hyperplanes. The group $W$ acts on the exterior algebra by $s\cdot
e_t=e_{sts}$ for $s\in S$ and $t\in T$. The ideal generated by elements of
the form (\ref{ARelation}) is homogeneous and $W$-stable, and so $A(W) =
\bigoplus_{p\ge 0} A^p(W)$ is a graded, skew-commutative $\BBC$-algebra on
which $W$ acts as algebra automorphisms. We denote the image of the
generator $e_t$ in $A(W)$ by $a_t$.

It is known that $A(W)$ is isomorphic to the cohomology ring $H^\bullet(M)$
as graded $W$-algebras \cite[Chapter 3]{orlikterao:arrangements} and that
$A^p(W)= H^p(M)=0$ for $p>n$.  We refer to $A^n(W)$ as the {\em top
  component} of $A(W)$.  It is shown in \cite{douglass:cohomology} that the
degree of the character $\omega_W^n$ afforded by $A^n(W)$ is the cardinality
of the set of cuspidal elements in $W$. It is known that the number of such
elements is the product $e_1\cdots e_n$, where $e_1, \ldots, e_n$ are the
exponents of $W$. Because $|W|= (e_1+1) \dotsm (e_n+1)$ we see that the
characters $\rho_W^n$ and $\omega_W^n$ have degree roughly equal $|W|/n$.

For a subset $J$ of $S$, let $X_J$ denote the set of minimal length right
coset representatives of $W_J$ in $W$ and let $x_J$ be the element
$\sum_{w\in X_J} w$ of the group algebra $\BBC W$. Solomon has shown that
the set $\{\, x_J\mid J\subseteq S\,\}$ is linearly independent and spans a
subalgebra of $\BBC W$ called the \emph{descent algebra} of $W$ (see
\cite{bergeronbergeronhowletttaylor:decomposition}).

Bergeron, Bergeron, Howlett, and Taylor
\cite[\S7]{bergeronbergeronhowletttaylor:decomposition} define a basis of
the descent algebra consisting of quasi-idempotents as follows. For subsets
$J$ and $K$ of $S$ define
\begin{align}\label{eq:mKJ}
  m_{KJ}= \left| \{\, x\in X_J\mid x^{-1}Jx\subseteq K\,\} \right|
  \quad\text{if}\quad J\subseteq K \quad\text{and}\quad m_{KJ}=0
  \quad\text{if}\quad J\not\subseteq K.
\end{align}

Note that $m_{KK}>0$ since $1_W\in X_K$ for all $K\subseteq S$.  Then the
$2^n\times 2^n$ matrix with rows and columns indexed by the power set of $S$
in some fixed non-increasing order and with $(K,J)$-entry $m_{KJ}$ is
lower-triangular with non-zero diagonal entries, and hence invertible.
Define $n_{KJ}$ to be the $(K,J)$-entry of the inverse matrix and put
$e_K=\sum_J n_{KJ} x_J$. Then $e_Ke_K= \gamma_Ke_K$, where $\gamma_K=|\{\,
L\subseteq S\mid \exists w\in W,\, w^{-1}Lw=K\,\}|$ and so each $e_K$ is a
quasi-idempotent in $\BBC W$. In particular, $e_S=\sum_J n_{SJ} x_J$ is an
idempotent.  By analogy with $A^n(W)$ we call $\BBC We_S$ the \emph{top
  component} of $\BBC W$ and denote the character it affords by $\rho_W^n$.
It is shown in \cite{bergeronbergeronhowletttaylor:decomposition} that the
degree of $\rho_W^n$ is the cardinality of the set of cuspidal elements in
$W$.

\begin{remark}\label{Irreducible}
  If $W=W_1\times W_2$ is reducible, then an element $(w_1,w_2)$ in
  $W_1\times W_2$ is cuspidal if and only if $w_1$ is cuspidal in $W_1$ and
  $w_2$ is cuspidal in $W_2$. It is straightforward to show that the
  idempotent generating the top component of $\BBC W$ is the product of the
  idempotents generating the top components of $\BBC W_1$ and $\BBC
  W_2$. Therefore, the top component of $\BBC W$ is isomorphic to the tensor
  product of the top components of $\BBC W_1$ and $\BBC W_2$.  Similarly,
  the top component of $A(W)$ is isomorphic to the tensor product of the top
  components of $A(W_1)$ and $A(W_2)$ by the K\"unneth theorem.
\end{remark}

\subsection{Main Results}
\label{sec:main-results}

In this subsection we state the main result (Theorem~\ref{thm:main}) of this
paper and show how it leads to a proof of Theorem~\ref{thm:all}.

\begin{theorem}\label{thm:main}
  Conjecture~B holds for all finite Coxeter groups of rank $n\le 4$.
\end{theorem}

Observe that if $W=W_1\times W_2$ is reducible, then since induction
commutes with outer tensor products, Remark~\ref{Irreducible} implies that
Conjecture~B holds for $W$ if and only if it holds for both $W_1$ and $W_2$
where the characters $\varphi_w$ satisfying Conjecture~B for $W$ are the
tensor products of those satisfying the conjecture for $W_1$ and $W_2$. Thus
it suffices to consider the case when $W$ is irreducible.

Conjecture~B has been proved with no restriction on the rank of $W$ for
symmetric groups in \cite{douglasspfeifferroehrle:coxeter} and dihedral
groups in \cite{douglasspfeifferroehrle:inductive}. In this paper we prove
the conjecture for the remaining finite irreducible Coxeter groups of rank
three or four in \S\ref{sec:results} by explicitly computing the linear
characters $\varphi_w$.  A description of the \GAP\ programs used in this
calculation is given in \S\ref{ImplementationSection}.

To prove Conjecture~1 for a particular Coxeter group $W$, we require a
linear character of the centralizer of a representative of every conjugacy
class of~$W$.  For cuspidal classes, we can use the characters satisfying
Conjecture~B.  For non-cuspidal conjugacy classes, these characters are
provided by Conjecture~C in the manner which we now describe.

Let $L$ be a subset of $S$ of size $r$.  Then $W_L$ acts on the top
components of $A(W_L)$ and $\BBC W_L$ and we denote the characters of $W_L$
afforded by these spaces by $\omega_L^r$ and $\rho_L^r$ rather than by
$\omega_{W_L}^r$ and $\rho_{W_L}^r$ to simplify notation.  Suppose that
Conjecture~B holds for $W_L$. Then for each $w$ in a set $\mathcal{C}_L$ of
representatives of the cuspidal conjugacy classes of $W_L$ we have a linear
character $\varphi_w$ of $C_{W_L}(w)$ such that
\[
\rho_L^r=\sum_{w\in\mathcal{C}_L} \Ind_{C_{W_L} (w)}^{W_L} \varphi_w
=\epsilon \omega_{L}^r.
\]
It is shown in \cite{konvalinkapfeifferroever:centralizers} that if $w$ is a
cuspidal element in $W_L$, then $C_W(w)$ is contained in $N_W(W_L)$ and the
quotient groups $C_W(w)/C_{W_L}(w)$ and $N_W(W_L)/W_L$ are isomorphic. It is
shown in \cite{douglasspfeifferroehrle:coxeter} that the characters
$\rho_L^r$ and $\omega_{L}^r$ of $W_L$ extend to characters
$\widetilde{\rho_L^r}$ and $\widetilde{\omega_L^r}$ of $N_W(W_L)$. Then
Conjecture~C asserts that there is a corresponding extension of each
character $\varphi_w$ to $C_W(w)$ for which a corresponding statement about
$\widetilde{\rho_L^r}$ and $\widetilde{\omega_L^r}$ holds.  The characters
and subgroups in Conjecture~C are summarized in the following diagram.
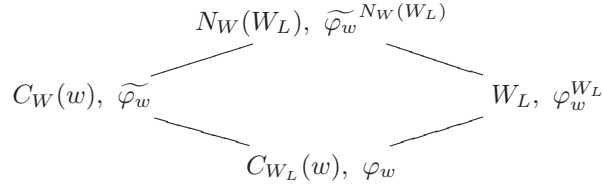
\[\xymatrix@=10pt{ &N_W(W_L),\ \widetilde{\varphi_w}^{N_W(W_L)} \ar@{-}[dl]
  \ar@{-}[dr] &  \\
  C_W(w),\ \widetilde{\varphi_w} \ar@{-}[dr]&& W_L,\ \varphi_w^{W_L}
  \ar@{-}[dl] \\
  &C_{W_L}(w),\ \varphi_w&}\]
The statement that Conjecture~C holds for all cases occurring in
Theorem~\ref{thm:all} is the following.

\begin{theorem}\label{thm:rel}
  Suppose that $W$ is a finite Coxeter group of rank at most four and that
  $L$ is a proper subset of $S$.  Then Conjecture~C holds for the pair $(W,
  W_L)$.
\end{theorem}

\begin{proof}
  If $W=W_1\times W_2$ is reducible, then by Remark~\ref{Irreducible} and
  the comments following Theorem~\ref{thm:main}, Conjecture~C holds for $(W,
  W_L)$ if and only if it holds for each of $(W_1,W_{L_1})$ and $(W_2,
  W_{L_2})$, where $L_1,L_2\subseteq S$ are such that $W_L=W_{L_1} \times
  W_{L_2}$.  Thus, we may assume that $W$ is irreducible.

  Recall that $W_L$ is said to be \emph{bulky} in $W$ if it has a normal
  complement in $N_W(W_L)$ (see \cite{bulky}). Conjecture~C is
  proved for any $W$ in \cite{douglasspfeifferroehrle:inductive} if $|L|\leq
  2$ or if $W_L$ is bulky. It follows that the conjecture holds if $W$ has rank
  three, so we may assume that $W$ has rank four and that $W_L$ is non-bulky
  of maximal rank.  The bulky parabolic subgroups of all finite irreducible
  Coxeter groups are listed in Appendix~\ref{a:bulky}. The eight pairs
  $(W,W_L)$ where $W$ has rank four and $W_L$ is non-bulky of maximal rank
  are listed in the following table.
  \[
  \renewcommand{\arraystretch}{1.4}
  \begin{array}{c|l}
    W&W_L\\\hline
    B_4&\widetilde{A}_1A_2\quad A_3 \\ 
    D_4&A_3\\
    F_4&A_1\widetilde{A}_2\quad A_2\widetilde{A}_1\\ 
    H_4&A_3 \quad A_1A_2 \quad I_2(5)A_1\\
  \end{array}
  \]
  In this table $\widetilde{A}_1$ and $\widetilde{A}_2$
  denote subgroups of types $A_1$ and $A_2$ generated by reflections
  orthogonal to short roots.

  With no restriction on the rank of $W$, it is shown
  in~\cite[\S7]{douglasspfeifferroehrle:coxeter} that Conjecture~C
  holds in case all factors of $W_L$ are of type $A$. This only leaves the
  pair $(H_4, I_2(5)A_1)$ to be considered. It is straightforward to verify
  that Conjecture~C holds in this case.
\end{proof}

Assuming the validity of Theorem~\ref{thm:main}, we can now complete the
proof of Theorem~\ref{thm:all}. It is shown in \cite[Theorem
4.7]{douglasspfeifferroehrle:inductive} that the characters $\varphi_w$
satisfying Conjecture~1 can be taken to be the characters $\varphi_w$
satisfying Conjecture~B together with the union of the sets of characters
$\widetilde{\varphi_w}$ satisfying Conjecture~C for the pair $(W,W_L)$ as
$W_L$ ranges over a set of representatives of the conjugacy classes of
proper parabolic subgroups of $W$.  If $W$ has rank at most four, these
characters are provided by Theorems~\ref{thm:main} and~\ref{thm:rel}.
Therefore, Conjecture~1 holds for all finite Coxeter groups of rank at most
four.

As stated above, Theorem~\ref{thm:main} is proved by direct computation. In
the next section we describe the algorithms used in this computation and in
\S\ref{sec:results} we present the results.

\section{Implementation}\label{ImplementationSection}

The proof of Theorem~\ref{thm:main} for each irreducible Coxeter group of
rank three or four consists of computing characters
$\varphi_w$ so that the conclusion of Conjecture~B holds.
In this section we describe how the top component
characters $\rho_W^n$ and $\omega_W^n$ as well as the characters
$\varphi_w$ were calculated.

The calculations were performed using the computer algebra system \GAP\
\cite{gap3}. The \CHEVIE\ package for \GAP\ \cite{chevie} provides
additional functionality for calculating with finite Coxeter groups $W$. In
particular, \CHEVIE\ provides the reflection representation of $W$ through
the matrices giving the action of the Coxeter generators on a vector space.

Recall that an element $w$ is cuspidal if none of its eigenvalues equals
$1$. Thus, to determine whether $w$ is cuspidal, we only need to inspect the
eigenvalues of the matrix representing $w$ in the reflection representation.
The cuspidal classes can also be determined using the \verb!CuspidalClasses!
function supplied by the \ZigZag\ package \cite{zigzag}.

\subsection{Computing the top component character of
  \texorpdfstring{$\mathbb{C}W$}{CW}}

Recall from \S\ref{sec:prelim} that the top component character of $\BBC W$
is the character afforded by the left ideal $\BBC W e_S$, where
$e_S=\sum\limits_{J\subseteq S}n_{SJ}x_J$
and the coefficients $n_{SJ}$ are entries of the matrix
$(n_{KJ})= (m_{KJ})^{-1}$, where $m_{KJ}$ is as in
equation~\eqref{eq:mKJ}.

Let $C_1,C_2,\ldots,C_m$ be the conjugacy classes of $W$ and let $w_i\in
C_i$ for $1\le i\le m$.  By Exercise~16 of \cite[\S9]{curtis_reiner_I} we
have
\[
\rho_W^n(w_i)=|C_W(w_i)|\sum_{w\in C_i}a_w\qquad\text{where}\qquad
e_S=\sum_{w\in W}a_ww.
\] 
Here we have used the fact that $\rho_W^n(w_i)=\rho_W^n(w_i^{-1})$, since
every element of a finite Coxeter group is conjugate to its inverse.  For
$w$ in $W$ set $\mathcal{D}(w) = \{\,s\in S\mid \ell(sw)<\ell(w)\,\}$. Then
$X_J= \{\, w\in W\mid \mathcal D(w) \subseteq S\setminus J\,\}$ and
$a_w=\sum\limits_{\mathcal{D}(w)\subseteq S\setminus J}n_{SJ}$. Therefore,
\begin{align} 
\label{eq:rhon} 
\rho_W^n(w_i)=|C_W(w_i)|\sum_{J\subseteq S} n_{SJ} \left|C_i \cap X_J\right|
\end{align}
and so the character $\rho_W^n$ can be computed using the following procedure:

\smallskip
\begin{enumerate}
\item Compute the coefficients $n_{SJ}$ for all $J\subseteq S$.
\item Compute $|C_i \cap X_J|$ for all $1\le i\le s$ and all $J \subseteq S$.
\item Compute $\rho_W^n(x_i)$ using equation~\eqref{eq:rhon} for all $1\le
  i\le m$.
\end{enumerate}

\smallskip 

\GAP\ and \CHEVIE\ provide functions for computing conjugacy classes and
centralizers in Coxeter groups. 
Using the \ZigZag\ package,
the matrix $(m_{KJ})= (n_{KJ})^{-1}$ can be
calculated with the method call \verb!Call(DescentAlgebra(W), "Mu")!, while
the sets $\mathcal D(w)$ and $X_J$ can be calculated using the
functions \verb!DescentClass!  and \verb!ParabolicTransversal!.
The sizes of the intersections $|C_i \cap X_J|$ for all $1\le i\le
m$ and all $J \subseteq S$ can be calculated
by looping once over the elements of $W$.

\subsection{Computing the top component character of
  \texorpdfstring{$A(W)$}{A(W)}}

We calculate the character $\omega_W^n$ by explicitly calculating the trace
of the representing matrices of $W$ on the top component of $A(W)$. To
perform this calculation, we need (1) a basis of the top component and (2) a
method for writing an arbitrary product $a_{t_1}a_{t_2}\cdots a_{t_p}$ in
$A(W)$ as a linear combination of basis elements in the given basis. We use
the so-called {\em non-broken circuit basis} of $A(W)$ described below.
The method may also be
applied to compute the character of the representation of $W$ on $A^p(W)$
for any $p$.

Recall that the generators $a_t$ of $A(W)$ are indexed by the reflections
$t\in T$. If $H_t$ denotes the $1$-eigenspace of $t$ for $t$ in $T$, then $t
\longleftrightarrow H_t$ describes a canonical bijection between $T$ and the set
of hyperplanes in the arrangement of $W$. In particular, the set $T$ indexes
the hyperplanes in the arrangement of $W$.

The following constructions apply more generally
to the Orlik-Solomon algebra $A(\mathcal{A})$ of any
essential hyperplane arrangement $\mathcal A=\{\, H_t\mid t\in T\,\}$
indexed by a finite set $T$ in an $n$-dimensional vector space. See \S3.1 of
\cite{orlikterao:arrangements} for more information about the algebras
$A(\mathcal{A})$.

\subsubsection{The non-broken circuit basis of the Orlik-Solomon algebra of
  an arrangement}

Let $H$ be a sequence $H_{t_1},H_{t_2},\ldots,H_{t_p}$ of hyperplanes. Then
$H$ is called a {\em circuit} if $H$ is dependent and
$H_{t_1},\ldots,\widehat{H_{t_j}},\ldots,H_{t_p}$ is independent for each
$1\le j\le p$. Note that \textsf{GAP} can easily test whether a tuple of
vectors is linearly independent using linear algebra functions such as
\verb!Rank!. Thus, once a set of linear functionals defining the hyperplanes
$H_t$ has been fixed, it is possible to test whether $H$ is a circuit.

Now fix a total order on $T$ and suppose that the sequence $H$ satisfies
$t_1<t_2<\cdots<t_p$. Then we call $H$
a {\em broken circuit} if $H_{t_1},H_{t_2},\ldots,H_{t_p},H_t$ is
  a circuit for some hyperplane $H_t$ with $t>t_p$.
We call $H$ a {\em non-broken circuit} if no subsequence of $H$ is a broken
  circuit.
Notice that the empty sequence is a non-broken circuit. It is shown in
Theorem~3.43 of \cite{orlikterao:arrangements} that
\begin{equation}\label{NBCBasis}
  \mathcal{B}=\left\{a_{t_1}a_{t_2}\cdots a_{t_p} \mid
    \text{$H_{t_1},H_{t_2},\ldots,H_{t_p}$ is a non-broken circuit}\right\} 
\end{equation}
is a basis of $A(\mathcal A)$. By construction $\mathcal{B}\cap
A^p(\mathcal A)$ is a basis of $A^p(\mathcal A)$ for $1\leq p\leq n$.
Therefore $\mathcal{B}\cap A^n(\mathcal A)$ is a basis of the top component of
$A(\mathcal A)$.

\begin{remark}\label{NBCIndependent}
  If $H$ is dependent, then any minimal dependent subsequence of $H$ is a
  circuit. Removing the last term of such a subsequence results in a broken
  circuit. This also shows that a non-broken circuit is independent.
\end{remark}

Observe that a sequence contains a broken circuit if and only if it contains
a {\em minimal} broken circuit.  Therefore, to identify the non-broken
circuits, it is sufficient to identify the set of minimal broken circuits.
We identify the minimal broken circuits together with the non-broken
circuits in Algorithm~NBC below. In the algorithm and in the following
justification, we denote by $H,H_t$ the sequence of hyperplanes obtained
from the sequence $H$ by appending the hyperplane $H_t$.

\medskip
\noindent{\bf Algorithm NBC} {\sl (Non-broken circuit basis)} This algorithm
takes as input a set $\mathcal A=\{\, H_t\mid t\in T\,\}$ of hyperplanes
indexed by a totally ordered set $T$. It returns a non-broken circuit basis
$\mathcal B$ of the Orlik-Solomon algebra $A(\mathcal A)$ and the set
$\mathcal M$ of minimal broken circuits.

{\sfcode`;=3000 \smallskip 

  {\bf initialize {$\mathcal{B}\gets\emptyset$, $\mathcal{M}\gets\emptyset$,
      $Q\gets\{()\}$}}

{\bf while $Q\ne\emptyset$ do}

\quad {remove the first element $H = H_{t_1}, H_{t_2}, \ldots, H_{t_p}$ of $Q$}

\quad {\bf if {\rm $H,H_t$ is independent for all $t>t_p$} then}

\qquad{append $H$ to $\mathcal{B}$}

\qquad{append $H,H_t$ to $Q$ for every $t$ such that $t>t_p$}

\quad{\bf else}

\qquad{append $H$ to $\mathcal{M}$ if $H$ does not contain any sequence in
  $\mathcal{M}$ as a subsequence}

\quad{\bf endif}

{\bf done}

{\bf return {$\mathcal{B},\mathcal{M}$}}.\bigskip}

Algorithm~NBC is justified as follows. Throughout the procedure we maintain
a queue $Q$ of sequences that remain to be considered. Initially $Q$
contains only the empty sequence. The procedure terminates when $Q$ is
empty. At every step $Q$ has the following properties.
\begin{enumerate}
\item\label{Increasing} Each sequence $H_{t_1}, H_{t_2}, \ldots, H_{t_p}\in
  Q$ satisfies $t_1<t_2<\cdots<t_p$.
\item\label{Independent} Each sequence $H\in Q$ is independent.
\item\label{ShorterCircuits} If the first sequence in $Q$ has length $p$,
  then all minimal broken circuits and all non-broken circuits of length
  less than $p$ have been discovered.
\end{enumerate}

In Algorithm~NBC suppose that the sequence $H$ removed from $Q$ is such that
$H, H_t$ is dependent for some $t>t_p$. Then $H, H_t$ contains a circuit as
a subsequence by Remark~\ref{NBCIndependent}.  At this point, all the minimal
broken circuits that are shorter than $H$ have been discovered by
(\ref{ShorterCircuits}), so if none of these is a subsequence of $H$, then
the entire sequence $H, H_t$ must be a circuit so that $H$ is a minimal
broken circuit.

On the other hand, suppose that $H, H_t$ is independent for all $t>t_p$.
Then each of the sequences $H, H_t$ for $t>t_p$ is appended to $Q$.  Note
that $H,H_t$ satisfies conditions (\ref{Increasing}) and (\ref{Independent})
above because $t>t_p$, while (\ref{ShorterCircuits}) holds because sequences
are always added to the end of $Q$ and hence appear after all the sequences
of shorter length.  At this point $H$ is added to $\mathcal{B}$.  This is
justified by the observation that if on the contrary, $H$ contains a broken
circuit $H'=H_{t_{j_1}},H_{t_{j_2}},\ldots,H_{t_{j_q}}$ as a subsequence,
then $H_{t_1}, H_{t_2}, \ldots, H_{t_{j_q}}$ (and hence $H$) would never
have been added to $Q$, since $H_{t_1}, H_{t_2}, \ldots, H_{t_{j_q}},H_t$ is
dependent for some $t>t_{i_q}$, namely the $t$ for which $H_t$ completes
$H'$ to a circuit.

It remains to show that if $H = H_{t_1}, H_{t_2}, \ldots, H_{t_p}$ is any
non-broken circuit (or respectively, any minimal broken circuit), then $H$
will be added to $\mathcal{B}$ (or respectively, to $\mathcal{M}$) at some
stage.  For this, it suffices to show that $H$ will be added to $Q$ at some
point, because as shown above, $H$ is identified as a non-broken circuit (or
respectively, a minimal broken circuit) when it is removed from
$Q$.  Observe that for all $q<p$ and all $t>t_q$ the sequence $H_{t_1},\ldots,
H_{t_q}, H_t$ is independent, since otherwise $H$ would contain a broken
circuit by Remark~\ref{NBCIndependent}.  Therefore, each of the sequences
$H_{t_1},\ldots, H_{t_q},H_t$ is added to $Q$.  In particular each sequence
$H_{t_1},\ldots, H_{t_q},H_{t_{q+1}}$ with $q<p$ is added to $Q$, so $H$
itself is added to $Q$.

\subsubsection{Expressing a product as a linear combination of basis
  elements} 

We now consider how to express a product $a=a_{t_1}a_{t_2}\cdots a_{t_p}$ of
$A(\mathcal A)$ in terms of the basis $\mathcal{B}$ returned by
Algorithm~NBC. If $a$ is not an element of the basis $\mathcal{B}$ then
$H=H_{t_1},\ldots,H_{t_p}$ must contain a broken circuit as a subsequence.
More specifically, it must contain a minimal non-broken circuit
$H'\in\mathcal{M}$, which we can identify directly by comparing all the
elements of $\mathcal{M}$ to $H$. Since $H'$ is a broken circuit, we can
find a hyperplane $H_t$ such that $H',H_t$ is a circuit.  The possibility
arises that $H_t$ is already an element of $H$. In this situation $a=0$
since $H$ contains the dependent subsequence $H',H_t$. Otherwise we use
(\ref{ARelation}) to express $a$ in terms of elements corresponding with
lexicographically larger sequences of the same length. Namely, multiplying
both sides of
\[
(-1)^{q+1} a_{t_{j_1}}\dotsm a_{t_{j_q}} =\sum_{k=1}^q(-1)^k a_{t_{j_1}}
\dotsm \widehat{a_{t_{j_k}}} \dotsm a_{t_{j_q}}a_t
\] 
by the remaining factors of $a$ results in $a$ on the left side and elements
of $A(\mathcal A)$ on the right side which can be expressed in terms of
$\mathcal{B}$ by induction.

It is straightforward to implement the recursive procedure just described
using standard functions in \GAP.

\subsection{Choosing linear characters of the centralizers}

Let $\{c_1,c_2,\ldots,c_s\}$ be a list of representatives of the cuspidal
classes of $W$. To find linear characters $\varphi_{c_i}$ of $C_W(c_i)$ that
satisfy the first equality in Conjecture~B, we proceed in three steps.
\begin{enumerate}
\item Find the linear characters $\varphi_{c_i, j}$ of the centralizer
  $C_W(c_i)$ for $1\leq i\leq s$.
\item Find the decomposition of each induced character $\varphi_{c_i,j}^W$
  as a sum of irreducible characters of $W$.
\item Use the exact packing algorithm described below to find a set $\{j_1,
  \dots, j_s\}$ so that $\sum_{i=1}^s \varphi_{c_i, j_i} = \rho_n^W$.
\end{enumerate}

\GAP\ and \CHEVIE\ provide tools to compute linear characters and to
decompose induced representations. To complete the third step we exploit the
fact that the irreducible characters of $W$ form a basis of its character
ring to convert the problem into an {\em exact packing problem}.
Namely, the character $\chi$ of a finite-dimensional
representation of $W$ can be
uniquely expressed as a sum $\chi= \sum_{j=1}^m n_j \chi_j$, where $\chi_1$,
\dots, $\chi_m$ are the irreducible characters of $W$ and $n_j$ is the
multiplicity of $\chi_j$ as a constituent of $\chi$, given by the intertwining
number $\langle \chi, \chi_j\rangle_W$. Thus, the top component character
$\rho_W^n$ of $\BBC W$ is given by an $m$-tuple of multiplicities, say
$g=(g_1, \dots, g_m)\in\mathbb{N}^m$, 
where $\mathbb{N}$ denotes the set of non-negative integers.
Similarly, for each $i$
the induced character $\Ind_{C(c_i)}^W \varphi_{c_i,j}$ is determined
by an $m$-tuple of multiplicities.
Let $L_i\subseteq\mathbb{N}^m$ be the set of these vectors.
Therefore, to find
characters $\varphi_{c_i,j_i}$ such that $\rho_W^n= \sum_{i=1}^s
\varphi_{c_i,j_i}^W$ we need  to enumerate the elements of the set
\begin{align*}
  E(g; L_1, \dots, L_s) =\left\{(\ell_1, \dots, \ell_s)
\in L_1 \times \dots \times L_s \mid \sum_{i=1}^s \ell_i = g\right\}.
\end{align*}
This exact packing problem is solved for $s > 0$
recursively by the formula
\begin{align*}
  E(g; L_1, \dots, L_s) = \coprod_{\substack{\ell_1\in L_1\\\ell_1\leq g}}
\left\{(\ell_1,\ell_2, \dots, \ell_s)
\mid (\ell_2, \dots, \ell_s) \in E(g - \ell_1; L_2, \dots, L_s)\right\}.
\end{align*}
Here we write $x\le y$ for $x,y\in\mathbb{N}^m$
if $x_i\le y_i$ for all $1\le i\le s$
where $x=(x_1,x_2,\ldots,x_m)$ and $y=(y_1,y_2,\ldots,y_m)$.
When $s = 0$, note that $E(g;)$ is the empty set if $g \neq 0$,
while $E(g;)$ consists of the empty sequence if $g = 0$.
The set $E(g; L_1, \dots, L_s)$ corresponds with a set of paths
in a labeled, rooted tree with labeled edges as follows.
The nodes and edges of this tree are labeled by vectors
in $\mathbb N^m$. The root of the tree is the goal vector $g$. If $i>0$
and $v$ is the label of a node at distance $i-1$ from the root, then the
children of this node 
are labeled by the vectors $v-\ell_i$ such that $\ell_i\in L_i$
and $v-\ell_i\geq 0$. The directed edge from $v$ to
$v-\ell_i$ is labeled by $\ell_i$. Then the leaves
at distance $s$ from the root and labeled by the zero vector are in
bijection with $E(g; L_1, \dots, L_s)$.
The solution $(\ell_1, \dots, \ell_s)$ corresponding to such a leaf is 
simply the sequence of labels on the path from the root to that leaf. 

For each permutation $\sigma$ of $\{1, \dots, s\}$ there is a natural bijection
between the sets $E(g; L_1, \dots, L_s)$ and $E(g; L_{\sigma(1)}, \dots, L_{\sigma(s)})$.  However, the trees corresponding to these sets can differ in size,
and thus be more or less expensive to compute.
The \ZigZag\ package contains an
efficient implementation of the construction just described in the form of the
function \texttt{ExactPackings}.

As an example, in the group $W(F_4)$ there are $s=9$ cuspidal conjugacy classes
and the lists $L_1$, \dots, $L_s$ contain
 $4$, $8$, $3$, $8$, $5$, $5$, $6$, $12$, $4$
elements.
The \ZigZag\ implementation finds the unique solution
by constructing a tree with $24$ vertices,
a considerable improvement over testing each of the approximately
$5,500,000$ elements of $L_1\times\cdots\times L_s$.

\section{Proof of Theorem~\ref{thm:main}} 
\label{sec:results}

In this section we present the results of our computations for the Coxeter
groups of types $B_3$, $H_3$, $B_4$, $D_4$, $F_4$, and $H_4$, thus verifying
Theorem~\ref{thm:main} for irreducible Coxeter groups of rank three and
four. In particular, for each element $w$ in a set of representatives of the
cuspidal conjugacy classes we find a linear character $\varphi_w$ of
$C_W(w)$ so that the conclusion of Conjecture~B holds.

A general method for constructing linear characters of centralizers is as
follows.  For $w$ in $W$ and $\zeta$ an eigenvalue of $w$ on $V$, let
$E(\zeta)$ denote the $\zeta$-eigenspace of $w$. Then $C_W(w)$ acts on
$E(\zeta)$ and $y\mapsto \det (y|_{E(\zeta)})^p$ defines a linear character
of $C_W(w)$ for each natural number $p$. Denote this character of $C_W(w)$
by $(\det|_{E(\zeta)})^p$.

It turns out that many of the characters $\varphi_w$ satisfying the
conclusion of Conjecture~B arise from the construction just
described. For example, if $w$ is a regular element in $W$ and $E(\zeta)$ is a
regular eigenspace of $w$, that is, such that $E(\zeta) \not \subseteq H_t$
for all $t$ in $T$, then with one exception, the character $\varphi_w$ is
equal $(\det|_{E(\zeta)})^p$ for some $p>0$. The exception is the class
labeled by the partition $22$ in type $B_4$ (see \S \ref{B4DataSection}).

We say that a conjugacy class in $W$ is \emph{regular} if it contains a
regular element. In all the groups we consider below, the longest element
$w_0$ is central and the character $\varphi_{w_0}$ is the sign character of
$W$. Thus $w_0$ is regular and $\varphi_{w_0}= \det|_{E(-1)}$. The Coxeter
class is well-known to be a regular class. If $w$ is a Coxeter element, then
it acts on its eigenspace $E(\zeta)$ as a cyclic group of order $|w|$, where
$\zeta$ is a primitive $|w|^{\text{th}}$ root of unity. It turns out to
always be the case that $\varphi_{w}= (\det|_{E(\zeta)})^p$, but it can
happen that $p\ne1$.

The rest of this section is devoted to giving the following information for
each finite Coxeter group of rank $3$ or $4$.
\begin{enumerate}
\item For each $w$ in a set of representatives of the cuspidal conjugacy
  classes of $W$,
  \begin{itemize}
  \item a generating set $G_w$ for $C_W(w)$, 
  \item a character $\varphi_w$ of $C_W(w)$, and
  \item the character values $\varphi_w(g)$, for $g$ in $G_w$.
  \end{itemize}
\item A table containing the values of the characters $\rho_W^n$ and
  $\omega_W^n$, as well as $\Ind_{C_W(w)}^W \varphi_w$ for each
  representative $w$.  In all cases we see that
  \[
  \rho_W^n=\sum_{w\in\mathcal{R}} \Ind_{C_W(w)}^W \varphi_w =\epsilon
  \omega_W^n
  \]
  as asserted in Theorem~\ref{thm:main}.  In these tables, the rows are
  indexed by the characters $\Ind_{C_W(w)}^W \varphi_w$ (denoted simply by
  $\varphi_w$), $\rho_W^n$, and $\omega_W^n$, and the columns are indexed by
  the conjugacy classes of $W$.
\item When $w$ is regular we identify the complex reflection group given by
  the action of $C_W(w)$ on a regular eigenspace $E(\zeta)$ and when
  possible, we compare the character $\varphi_w$ with $\det|_{E(\zeta)}$.
\end{enumerate}

The following notation is in force. The cyclic group of size $n$ is denoted
by $\cyclic{n}$ and the symmetric group on $n$ letters is denoted by
$\symmetric{n}$. For $n\ge 1$ we denote the primitive complex
$n^{\text{th}}$ root of unity $e^{2\pi i/n}$ by $\zeta_n$.  As in the proof
of Theorem~\ref{thm:rel}, the labels $\widetilde{A}_1$ and $\widetilde{A}_2$
denote subgroups of types $A_1$ and $A_2$ generated by reflections
orthogonal to short roots.  The same convention applies to $\widetilde{D}_4$
in $W(F_4)$. We denote partitions as strings of numbers, without commas,
written in non-decreasing order. Finally, when $w$ lies in more than one
Coxeter group, for example $w\in W(B_3)\subseteq W(B_4)$, and we wish to
compare the characters of centralizers of $w$ in different groups, we
decorate the symbol $\varphi_w$ with a superscript giving the type of the
group in which $w$ lies.

\subsection{\texorpdfstring{$W$}{W} of type \texorpdfstring{$B$}{B}}
\label{BTheorySection}
Assume now that $V$ is the reflection representation of $W=W(B_n)$ with
$n\geq 2$ and let $\{\, v_1, \dots, v_n\,\}$ be a basis of $V$.  We view $W$
as acting on $V$ by signed permutations of $\{\,v_1,\ldots,v_n\,\}$.
Namely, the Coxeter generators $s_1,s_2\ldots,s_n$ are given by
\[
s_1(v_k)= \begin{cases} -v_1,&k=1\\\phantom{-}v_k,&k\ne1 \end{cases}\quad
\text{and} \quad s_i(v_k)= \begin{cases}\phantom{_{-1}}v_{i},&k=i-1\\
  v_{i-1},& k=i \\
  \phantom{_{-1}}v_k,&k\ne i-1,i \end{cases} \quad \text{for $i>1$.}
\] 
The Dynkin diagram of $W(B_n)$ is $\begin{xy}<.75cm,0cm>: (0,0)="1";
  (1,0)="2"**\dir{=}; ?*{<}; "2";(2,0)="3"**\dir{-}; (2.5,0)**\dir{-};
  (3,0)="4"; (3.5,0); (4,0)="5"**\dir{-}; (5,0)="6"**\dir{-};
  "1"*{\bullet}*+!U{_1}; "2"*{\bullet}*+!U{_2}; "3"*{\bullet}*+!U{_3};
  "4"*{\cdots}; "5"*{\bullet}*+!U{_{n{-}1}}; "6"*{\bullet}*+!U{_n};
\end{xy}$.  

For $1\leq i<j\leq n$ we define elements $t_i$ and $s_{i,j}$ by
\[
t_i(v_k)= \begin{cases} -v_i,&k=i\\\phantom{-} v_k,&k\ne i \end{cases}\quad
\text{and} \quad s_{i,j}(v_k)= \begin{cases} v_{j},&k=i\\ v_{i},& k=j \\
  v_k,&k\ne i,j. \end{cases}
\]

It is well-known that the conjugacy classes in $W(B_n)$ are indexed by
double partitions $\mu.\lambda$ of $n$. Namely, if
$\mu=\mu_1\mu_2\cdots\mu_q$ and $\lambda=\lambda_1\lambda_2\cdots\lambda_p$
are such that $\sum_{i=1}^q \mu_i + \sum_{j=1}^p \lambda_j =n$, then
elements of the conjugacy class indexed by $\mu.\lambda$ have $q$
``positive'' cycles of lengths $\mu_1,\ldots,\mu_q$ and $p$ ``negative''
cycles of lengths $\lambda_1,\ldots,\lambda_q$. If $\mu$ or $\lambda$ is the
empty partition, then it is omitted from the notation. With this labeling,
the cuspidal conjugacy classes are indexed by the double partitions of the
form $.\lambda$. Since we only consider cuspidal classes, we write $\lambda$
instead of $.\lambda$. See \cite[\S3.4] {geckpfeiffer:characters} for more
details.

Fix a partition $\lambda=\lambda_1\lambda_2\cdots\lambda_p$ of $n$. Set
$\tau_1=0$ and for $i>1$ define $\tau_i=\lambda_1+ \dots+
\lambda_{i-1}$. Then $\tau_{i}+ \lambda_i= \tau_{i+1}$ and
$\tau_{p+1}=n$. For $1\leq i\leq p$ define
\[
c_i= t_{\tau_{i}+1} s_{\tau_{i}+2} s_{\tau_{i}+3} \cdots s_{\tau_{i+1}}
\]
in $W$. Then $c_i$ has order $2\lambda_i$ and acts on the set
$\{v_{\tau_i+1}, \dots, v_{\tau_{i+1}} \}$ as a ``negative
$\lambda_i$-cycle.'' Define
\[
w_\lambda= c_1c_2\dotsm c_p.
\]
Then $w_\lambda$ is a representative of the cuspidal conjugacy class labeled
by $\lambda$. For each $i$ such that $\lambda_i= \lambda_{i+1}$ define
\[
x_i= s_{\tau_i+1, \tau_{i+1}+1} s_{\tau_i+2, \tau_{i+1}+2} \dotsm
s_{\tau_{i+1}, \tau_{i+2}}.
\]
It is straightforward to check that $x_i$ centralizes $w_\lambda$ and that
$C_W(w_\lambda)$ is generated by
\[
\{\, c_i\mid 1\leq i\leq p\,\} \cup \{\, x_i\mid 1\leq i\leq p,\, \lambda_i=
\lambda_{i+1}\,\}.
\]
If $\lambda$ has $m_i$ parts equal $i$, then $C_W(w_\lambda ) =\prod_{m_i>0}
\cyclic{2i} \wr \symmetric{m_i}$.

It follows from the definition of $w_\lambda$ that if $\lambda$ is the
partition with all parts equal $1$, then $w_\lambda$ is the longest element
$w_0$ of $W$.  If $\lambda$ is the partition with the single part $n$, then
$w_\lambda$ is a Coxeter element.  To simplify the notation, the character
$\varphi_{w_\lambda}$ of $C_W(w_\lambda)$ is denoted simply by
$\varphi_\lambda$.

\subsubsection{$W=W(B_3)$}\label{ss:B3}
The cuspidal conjugacy classes are labeled by the partitions $111$, $12$,
and $3$. Characters $\varphi_\lambda$ satisfying the conclusion of
Conjecture~B are given in the following table.  For each partition
$\lambda$, the table lists the isomorphism type of $C_W(w_\lambda)$ in the
second row, the generators of $C_W(w_\lambda)$, using the notation from
\S\ref{BTheorySection}, in the third row, and the values of
$\varphi_\lambda$ on the generators in the fourth row. When
$\varphi_\lambda$ is the sign character, we denote it by $\epsilon$ and omit
the character values.
\[
\renewcommand{\arraystretch}{1.2}
\begin{array}{c||ccc|cccc|cc}
  \lambda& \multicolumn{3} {c|} {111}& \multicolumn {4} {c|} {12}&
  \multicolumn {2} {c} {3}\\ \hline
  C_W(w_\lambda)&&W&&&\multicolumn{2}{c}{\cyclic 2 \times \cyclic 4}&&&
  \cyclic 6 \\ \hline 
  \text{Generators}&& S &&&c_1& c_2 &&& w_{3} \\ \hline
  \varphi_\lambda &&\epsilon &&& -1&-1&&&\zeta_6  
\end{array}
\]

The classes $111$ and $3$ are regular and $\varphi_{3}= \det|_{E(\zeta_6)}$.

The values of the characters $\varphi_\lambda^W$ together with $\rho_W^3$
and $\omega_W^3$ are given in Table~\ref{fig:B3}.

\begin{table}[h!tb]
  \[
  \renewcommand{\arraystretch}{1.2}
  \begin{array}{r|rrrrrrrrrr}
    &111.&11.1&1.11&.111&12.&1.2&2.1&.12&3.&.3\\\hline
    \epsilon=\varphi_{111}&1&-1&1&-1&-1&1&1&-1&1&-1\\ 
    \varphi_{12}&6&-2&2&-6&\cdot&-2&\cdot&2&\cdot&\cdot\\
    \varphi_3&8&\cdot&\cdot&-8&\cdot&\cdot&\cdot&\cdot&-1&1\\ \hline
    \rho_W^3& 15&-3&3&-15&-1&-1&1&1&\cdot&\cdot\\ 
    \omega_W^3&15&3&3&15&1&-1&1&-1&\cdot&\cdot
  \end{array}
  \]
  \caption{The characters $\varphi_\lambda^W$, $\rho_W^3$, and
    $\omega_W^3$ for $W(B_3)$}\label{fig:B3}
\end{table}

\subsubsection{$W=W(B_4)$}\label{B4DataSection}
The cuspidal conjugacy classes are labeled by the partitions of $4$. With
the same conventions as for $W(B_3)$, characters $\varphi_\lambda$
satisfying the conclusion of Conjecture~B are given in the following table.
\[
\renewcommand{\arraystretch}{1.3}
  \begin{array}{c||c|ccc|cc|cc|c}
    \lambda& 1111& \multicolumn {3} {c|} {112}&
    \multicolumn{2}{c|}{22} &\multicolumn{2} {c|} {13} &4\\\hline  
    C_W(w_\lambda)&W&\multicolumn{3}{c|}{(\cyclic{2}\wr \symmetric{2})
      \times \cyclic{4}}&\multicolumn{2}{c|} {\cyclic 4 \wr \symmetric
      2}& \multicolumn{2} {c|} {\cyclic 2\times \cyclic 6} &\cyclic 8\\ \hline 
    \text{Generators}&S&c_1&x_1&c_3&c_1&x_1&c_1&c_2&w_4 \\ \hline   
    \varphi_\lambda &\epsilon& -1&-1&-1& -1 &-1 & -1
    &\zeta_6 &-1 
  \end{array} 
\]

The regular classes are $1111$, $22$, and $4$. In contrast with the Coxeter
class in type $B_3$ where $\varphi_{3}= \det|_{E(\zeta_6)}$, in this case we
have $\varphi_{4}= (\det|_{E(\zeta_8)})^4$. It is easy to see that if
$\zeta$ is any primitive fourth root of unity, then $C_W(w_{22})$ acts on
the two-dimensional, regular eigenspace $E(\zeta)$ as the complex reflection
group $G(4,1,2) \cong \cyclic 4 \wr \symmetric 2$. The eigenvalues of $c_1$
and $x_1$ on $E(\zeta)$ are $\{1,\zeta\}$ and $\{1,-1\}$, respectively. Thus
$\det |_{E(\zeta)}$ takes the values $\zeta$ and $-1$ at $c_1$ and $x_1$
while $\varphi_{22}$ takes the value $-1$ at both $c_1$ and $x_1$.  It
follows that $\varphi_{22}$ is not equal to any power of $\det|_{E(\zeta)}$.

The values of the characters $\varphi_\lambda^W$ together with $\rho_W^4$
and $\omega_W^4$ are given in Table~\ref{fig:B4}.

\begin{table}[htbp]
  \small
  \begin{multline*}
    \renewcommand{\arraystretch}{1.2}
    \begin{array}{r|rrrrrrrrrr}
      &1111.&111.1&11.11&1.111&.1111&112.&11.2&12.1&
      1.12&2.11\\\hline
      \epsilon=\varphi_{1111}& 1&-1&1&-1&1&-1&1&1&-1&-1\\
      \varphi_{112}& 12&-6&4&-6&12&-2&\cdot&\cdot&2&-2\\
      \varphi_{22}& 12&\cdot&4&\cdot&12&\cdot&-4&\cdot&\cdot&\cdot\\
      \varphi_{13}& 32&-8&\cdot&-8&32&\cdot&\cdot&\cdot&\cdot&\cdot\\
      \varphi_4&
      48&\cdot&\cdot&\cdot&48&\cdot&\cdot&\cdot&\cdot&\cdot\\ \hline
      \rho_W^4&105&-15&9&-15&105&-3&-3&1&1&-3\\
      \omega_W^4&105&15&9&15&105&3&-3&1&-1&3
    \end{array}\\ \\
    \renewcommand{\arraystretch}{1.2}
    \begin{array}{rrrrrrrrrr|l}
      .112&22.&2.2&.22&13.&1.3&3.1&.13& 4.&.4&\\\hline
      1&1&-1&1&1&-1&-1&1&-1&1& \varphi_{1111}=\epsilon\\
      \cdot&\cdot&2&-4&\cdot&\cdot&\cdot&\cdot&\cdot&\cdot&
      \varphi_{112} \\ 
      -4&-4&\cdot&\cdot&\cdot&\cdot&\cdot&\cdot&\cdot&2& \varphi_{22}\\
      \cdot&\cdot&\cdot&\cdot&-1&1&1&-1&\cdot&\cdot& \varphi_{13}\\
      \cdot&\cdot&\cdot&8&\cdot&\cdot&\cdot&\cdot&\cdot&-4&
      \varphi_{4}\\ \hline 
      -3&-3&1&5&\cdot&\cdot&\cdot&\cdot&-1&-1& \rho_W^4\\
      -3&-3&-1&5&\cdot&\cdot&\cdot&\cdot&1&-1& \omega_W^4
    \end{array}
  \end{multline*}
  \caption{The characters $\varphi_\lambda^W$, $\rho_W^4$, and
    $\omega_W^4$ for $W(B_4)$}\label{fig:B4}
\end{table}

\subsection{\texorpdfstring{$W$}{W} of type
  \texorpdfstring{$D$}{D}} \label{DTheorySection} 

As above, $V$ has basis $\{\, v_1, \dots, v_n\,\}$. We consider elements of
$W=W(D_n)$ as acting as signed permutations of $\{\, v_1, \dots, v_n\,\}$
with even number of sign changes. Then $W(D_n)$ is a normal subgroup of
$W(B_n)$ of index $2$. The Coxeter generators of $W(D_n)$ are
$s_1',s_2,\dots,s_n$ where $s_2,\dots, s_n$ are the last $n-1$ Coxeter
generators of $W(B_n)$ defined in \S\ref{BTheorySection} and $s_1'$ is given
by
\[
s_1'(v_k)= \begin{cases}
  -v_2,&k=1\\
  -v_1, &k=2\\
  \phantom{-}v_k,&k\ne1,2.
\end{cases}
\]
Then $s_1'=s_1s_2s_1$, where $s_1$ is the first Coxeter generator of
$W(B_n)$. The Dynkin diagram of $W=W(D_n)$ is
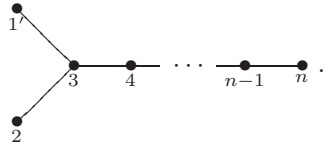
\[
\begin{xy}<.75cm,0cm>: (0,-1)="1"; (0,1)="2"; (1,0)="3"**\dir{-};
  "1";"3"**\dir{-}; (2,0)="4"**\dir{-}; (2.5,0)**\dir{-}; (3,0)="5";
  (3.5,0);(4,0)="6"**\dir{-}; (5,0)="7"**\dir{-}; "1"*{\bullet}*+!U{_2};
  "2"*{\bullet}*+!U{_{1'}}; "3"*{\bullet}*+!U{_3}; "4"*{\bullet}*+!U{_4};
  "5"*{\cdots}; "6"*{\bullet}*+!U{_{n{-}1}}; "7"*{\bullet}*+!U{_n};
  \end{xy} .
\]

Because $W(D_n)$ is a normal subgroup of $W(B_n)$, it is a union of
conjugacy classes of $W(B_n)$. The conjugacy class of $W(B_n)$ labeled by
the double partition $\mu.\lambda$ of $n$ lies in $W(D_n)$ if and only if
$\lambda$ has an even number of parts. Moreover, an element in $W(D_n)$ is
cuspidal if and only if it is cuspidal in $W(B_n)$ and cuspidal classes do
not fuse, so the cuspidal conjugacy classes of $W(D_n)$ are labeled by
partitions of $n$ with an even number of parts. For such a partition
$\lambda$ we take $w_\lambda$ to be the representative of the conjugacy
class of $W(B_n)$ chosen in \S\ref{BTheorySection}. Then the centralizer in
$W(D_n)$ of $w_\lambda$ is the intersection of $W(D_n)$ and the centralizer
of $w_\lambda$ in $W(B_n)$. See \cite[\S3.4]{geckpfeiffer:characters} for
more details.

\subsubsection{$W=W(D_4)$}\label{D4DataSection}
The cuspidal conjugacy classes are labeled by the partitions $1111$, $22$,
and $13$. All three classes are regular. Each of these conjugacy classes is
also a conjugacy class in the larger group $W(B_4)$, and as remarked above,
$C_W(w_\lambda)= W\cap C_{W(B_4)}(w_\lambda)$.  One might conjecture that
the character of $\varphi_{\lambda}^{D_4}$ of $C_{W(D_4)}(w_\lambda)$ is the
restriction of the character of $\varphi_{\lambda}^{B_4}$ of
$C_{W(B_4)}(w_\lambda)$. This turns out to be the case for the class of
$w_0$ labeled by $1111$ and for the Coxeter class labeled by $13$, but not
for the class labeled by $22$.

Let $w_{22}= s_1's_3 s_1' s_2 s_3 s_4$. Using the notation introduced in
\S\ref{BTheorySection} we have $w_{22}= c_1 c_2$. The centralizer of
$w_{22}$ in $W(D_4)$ contains the generator $x_1$ of the centralizer of
$w_{22}$ in $W(B_4)$, but not the generator $c_1$. Rather,
$C_{W(D_4)}(w_{22})$ is generated by $w_{22}$, $x_1$, and the involution
$s_1's_2$. The character $\varphi_{22}$ maps each of these generators to
$-1$. Notice that $\varphi_{22}^{B_4} (w_{22}) =\left(\varphi_{22}^{D_4}
  (w_{22})\right)^2$.

With the same conventions as for $W(B_3)$ , characters $\varphi_\lambda$
satisfying the conclusion of Conjecture~B are given in the following table.
\[
\renewcommand{\arraystretch}{1.2}
\begin{array}{c||ccc|ccccc|ccc}
  \lambda& \multicolumn{3} {c|} {1111}& \multicolumn {5} {c|} {22}&
  \multicolumn {3} {c} {13}\\ \hline
  C_W(w_\lambda)& \multicolumn{3} {c|} {W}& \multicolumn {5} {c|}
  {(\cyclic 4\times \cyclic 2)\rtimes \symmetric 2}&
  \multicolumn {3} {c} {\cyclic 6}\\ \hline
  \text{Generators}&& S &&&
  w_{22}& s_1's_2 &x_1  &&& w_{13} & \\ \hline
  \varphi_\lambda &&\epsilon&&& -1&-1 &-1&&&\zeta_3 & 
\end{array}
\]

For the Coxeter class $13$ we have $\varphi_{13}= (\det|_{ E(\zeta_6)}
)^2$. Let $\zeta$ be a primitive fourth root of unity. Using \cite[Theorem
4.2]{springer} it is easy to compute that if $\zeta$ is any primitive fourth
root of unity, then $C_W(w_{22})$ acts on the two-dimensional, regular
eigenspace $E(\zeta)$ as the complex reflection group $G(4,2,2) \cong
(\cyclic 4\times \cyclic 2)\rtimes \symmetric 2$. The eigenvalues of
$w_{22}$, $s_1's_2$, and $x_1$ acting on $E(\zeta)$ are $\{\zeta,\zeta\}$,
$\{1,-1\}$ and $\{1,-1\}$, respectively. Thus
$\varphi_{22}=\det|_{E(\zeta)}$.

The values of characters $\varphi_\lambda^W$ together with $\rho_W^4$ and
$\omega_W^4$ are shown in Table~\ref{fig:D4}.

\begin{table}[htbp]
  \[
  \renewcommand{\arraystretch}{1.2}\arraycolsep4pt
  \begin{array}{r|rrrrrrrrrrrrr}
    &1111.&11.11&.1111&112.&1.12&2.11&22.+&22.-&.22&13.&.13&4.+&4.-\\
    \hline 
    \epsilon=\varphi_{1111}&1&1&1&-1&-1&-1&1&1&1&1&1&-1&-1\\
    \varphi_{22}&12&-4&12&\cdot&\cdot&\cdot&-4&-4&4&\cdot&\cdot&\cdot&\cdot\\
    \varphi_{13}&32&\cdot&32&\cdot &\cdot&\cdot&\cdot &\cdot&\cdot&-1&-1
    &\cdot&\cdot\\ 
    \hline \rho_W^4&45&-3&45&-1&-1&-1&-3&-3&5&\cdot&\cdot&-1&-1\\  
    \omega_W^4&45&-3&45&1&1&1&-3&-3&5&\cdot&\cdot&1&1
  \end{array}
  \]
  \caption{The characters $\varphi_\lambda^W$, $\rho_W^4$, and  $\omega_W^4$
    for $W(D_4)$}\label{fig:D4} 
\end{table}

\subsection{\texorpdfstring{$W=W(F_4)$}{W=W(F4)}}
The Dynkin diagram of $W$ is $\vcenter{\begin{xy}<.75cm,0cm>: (0,0)="1";
    (1,0)="2"**\dir{-}; (2,0)="3"**\dir{=}; ?*{>}; "3";(3,0)="4"**\dir{-};
    "1"*{\bullet}*+!U{_1}; "2"*{\bullet}*+!U{_2}; "3"*{\bullet}*+!U{_3};
    "4"*{\bullet}*+!U{_4}; \end{xy}}$.  We label the conjugacy classes of
$W$ using Carter's labeling \cite{carter:conjugacy}. This is also the
labeling used by the \CHEVIE\ package.  There are nine cuspidal conjugacy
classes. Their Carter labels are
\[
4A_1,\quad D_4,\quad D_4(a_1),\quad C_3A_1,\quad A_2\widetilde{A}_2, \quad
F_4(a_1),\quad F_4,\quad A_3\widetilde{A}_1,\quad\text{and}\quad B_4.
\]
For each label $d$ above we denote the representative of the class labeled
$d$ by $w_d$ and the character of $C_W(w_d)$ satisfying
Theorem~\ref{thm:main} by $\varphi_d$.  The classes labeled by $4A_1$,
$F_4$, and $B_4$ are self-centralizing and we consider them first.

\medskip
\begin{description}
\item[$4A_1$] The class labeled by $4A_1$ is $\{w_0\}$. We take
  $\varphi_{4A_1} =\epsilon$.

\item[$F_4$] This is the Coxeter class. Coxeter elements have order
  $12$. We take $w_{F_4}$ to be any Coxeter element and define
  $\varphi_{F_4}$ by $\varphi_{F_4}(w_{F_4})=\zeta_3$. Clearly,
  $\varphi_{F_4}= (\det|_{E(\zeta_{12})})^4$.

\item[$B_4$] This class is regular and contains the Coxeter elements in the
  maximal rank subgroups of $W$ of type $B_4$. Then $C_W(w_{B_4})=
  C_{W(B_4)}(w_{B_4})$ is cyclic of order $8$, where $w_{B_4}$ denotes the
  element $w_4$ from \S\ref{B4DataSection}. Define $\varphi_{B_4}$ by
  $\varphi_{B_4} (w_{B_4})=\zeta_4$. We see from \S\ref{B4DataSection} that
  $\varphi_{B_4} ^{B_4}= \left(\varphi_{B_4} ^{F_4}\right)^2$. In addition,
  $\varphi_{B_4}=(\det|_{E(\zeta_8)})^2$.

\item[$D_4$] This class contains the Coxeter elements in the maximal rank
  subgroups of $W$ of type $D_4$ and $B_3\widetilde{A}_1$ (see
  \cite{douglasspfeifferroehrle:reflection}). It also contains the class
  labeled by the partition $13$ in the maximal rank subgroups of $W$ of
  types~$B_4$ and $D_4$.  In $W(D_4)$ the centralizer of $w_{13}$ is
  isomorphic to $\cyclic 6$ while in $W(B_3\widetilde{A}_1)$ and $W(B_4)$
  the centralizer of $w_{13}$ is isomorphic to $\cyclic 6 \times \cyclic 2$.

  Notice that because $w_0$ is central in $W$, multiplication by $w_0$
  permutes the conjugacy classes of $W$ and $C_W(ww_0)= C_W(w)$ for all
  $w\in W$. In this case, multiplication by $w_0$ sends the class labeled
  $D_4$ to the class labeled $A_2$ containing $s_1s_2$. Thus we can take
  $w_{D_4}=s_1s_2w_0$.  Extending the Dynkin diagram of $W$ as in the
  Borel-De Siebenthal algorithm ~\cite{Boreldesiebenthal:groupes} by
  adjoining the reflection $s_{21}$ corresponding to the highest short root
  results in the diagram
  \begin{equation}
    \label{diag:f4ext1}
    \vcenter{\begin{xy}<.75cm,0cm>: (0,0)="1"; (1,0)="2"**\dir{-};
        (2,0)="3"**\dir{=}; ?*{>}; "3";(3,0)="4"**\dir{-};
        (4,0)="5"**\dir{-}; "1"*{\bullet}*+!U{_1}; "2"*{\bullet}*+!U{_2};
        "3"*{\bullet}*+!U{_3}; "4"*{\bullet}*+!U{_4};
        "5"*{\bullet}*+!U{_{{21}}};
      \end{xy}}.
  \end{equation}
  The subgroup generated by $\{s_4,s_{21}\}$ is a parabolic subgroup of type
  $\widetilde A_2$ and
  \[
  C_W(w_{D_4}) =\langle s_4,s_{21}\rangle\times\langle w_{D_4}\rangle
  \cong\symmetric{3}\times\cyclic{6}.
  \] 
  We define $\varphi_{D_4}$ by $\varphi_{D_4}|_{\symmetric 3} =
  \epsilon_{\symmetric 3}$ and $\varphi_{D_4}(w_{D_4})= \zeta_3$. Notice
  that using \S\ref{B4DataSection} and \S\ref{D4DataSection} we have
  \[
  \zeta_3=\varphi_{D_4} ^{D_4} (w_{D_4})= \left(\varphi_{D_4} ^{B_4}
    (w_{D_4})\right)^2 =\varphi_{D_4} ^{F_4} (w_{D_4}).
  \]

\item[$D_4(a_1)$] This class is regular and contains the conjugacy classes
  labeled by the partition $22$ in the maximal rank subgroups of $W$ of
  types $B_4$ and $D_4$. It also contains the Coxeter elements in the
  maximal rank subgroup of $W$ of type $B_2B_2$.  To find a representative
  of this class and compute its centralizer we extend the Dynkin diagram of
  $W$ by adjoining the reflection $s_{24}$ corresponding to the highest long
  root. The resulting diagram is
  \[
  \vcenter{\begin{xy}<.75cm,0cm>: (-1,0)="0"; (0,0)="1"**\dir{-};
      "1";(1,0)="2"**\dir{-}; (2,0)="3"**\dir{=}; ?*{>}; "3";
      (3,0)="4"**\dir{-}; "0"*{\bullet}*+!U{_{{24}}}; "1"*{\bullet}*+!U{_1};
      "2"*{\bullet}*+!U{_2}; "3"*{\bullet}*+!U{_3}; "4"*{\bullet}*+!U{_4};
    \end{xy}}
  \]
  and consider the maximal rank subgroup $W(B_4)$ generated by $\{s_{24},
  s_1, s_2, s_3\}$. Recall from \S\ref{B4DataSection} that the centralizer
  of $w_{22}=$ in $W(B_4)$ is generated by $c_1$ and $x_1$. Translating from
  the $B_4$ labeling to our current labeling, we set
  \[
  w_{22}=s_3s_2s_1s_2s_3s_2s_1s_{24}, \quad c_1=s_3s_2,\quad\text{and}\quad
  x_1=s_1s_2s_{24}s_1.
  \] 
  Then $w_{22}$ lies in the conjugacy class we are considering and
  $C_{W(B_4)}(w_{22})= \langle c_1, x_1\rangle$.  Although $w_{22}$ would be
  a natural representative of $D_4(a_1)$, it is more convenient to define
  $w_{D_4(a_1)}$ to be the conjugate $s_1s_2s_1 w_{22} s_1s_2s_1$ of
  $w_{22}$ because then $w_{D_4(a_1)}$ will commute with the representative
  $w_{\widetilde{A}_2A_2}$ chosen below. Define $c_1'= s_1s_2s_1 c_1
  s_1s_2s_1$ and $x_1'= s_1s_2s_1 x_1 s_1s_2s_1$. Then $C_W(w_{D_4(a_1)})$
  is generated by $\{c_1', x_1', w_{\widetilde{A}_2A_2}\}$.

  Define $\varphi_{D_4(a_1)}$ by
  \[
  \varphi_{D_4(a_1)}(c_1')=-1,\quad \varphi_{D_4(a_1)}(x_1')=1, \quad
  \text{and}\quad \varphi_{D_4(a_1)}(w_{\widetilde{A}_2A_2})=1.
  \]
  Notice that using \S\ref{B4DataSection} and \S\ref{D4DataSection} we have
  \[
  \varphi_{D_4(a_1)} ^{F_4} (w_{D_4(a_1)})= \varphi_{22}^{B_4}(w_{22})
  =(\varphi_{22}^{D_4}(w_{22}))^2.
  \]

  Using \cite[Theorem 4.2]{springer} it is easy to see that
  $C_W(w_{D_4(a_1)})$ acts on its $\zeta_4$-eigenspace as the complex
  reflection group $G_8$. This group may be described by the diagram
  $\vcenter{\begin{xy}<.5cm,0cm>:
      (0,0)+<4pt,0pt>;(2,0)-<4pt,0pt>**\dir{-};?*+!D{_3};
      (0,0)*{_4}*\xycircle<4pt>{-}*++!U{_c};
      (2,0)*{_4}*\xycircle<4pt>{-}*++!U{_d} \end{xy}}$ where $c=c_1'$ and
  $d=s_2s_3s_4s_3$. Then $\varphi_{D_4(a_1)} (c)= \varphi_{D_4(a_1)} (d)=
  -1$. Also, $\varphi_{D_4(a_1)}= (\det|_{E(\zeta_{4})})^2$.

\item[$C_3A_1$] This class is regular and contains the Coxeter elements of
  the maximal rank subgroups of $W$ of types $C_3A_1$ and $\widetilde{D}_4$
  (see \cite{douglasspfeifferroehrle:reflection}). In particular, this class
  is the image under the graph automorphism of $W$ of the class labeled by
  $D_4$ so the centralizers of elements of both classes are isomorphic to
  $\cyclic 6 \times \symmetric 3$.  We put $w_{C_3A_1}=s_3s_4w_0$ and
  compute its centralizer using the same technique used for the class
  labeled by $D_4$. We define $\varphi_{C_3A_1}$ by
  $\varphi_{C_3A_1}(w_{C_3A_1})= \zeta_3$ and $\varphi_{C_3A_1}|_{\symmetric
    3} = \epsilon$.

\item[$A_2\widetilde{A}_2$] This class is regular and contains the Coxeter
  elements of the reflection subgroups of $W$ of type $A_2 \widetilde A_2$.
  We take $w_{A_2\widetilde{A}_2}=s_1s_2s_4s_{21}$ (see
  (\ref{diag:f4ext1})). We noted above that $w_{D_4(a_1)}$ was chosen so
  that $w_{A_2\widetilde{A}_2}$ and $w_{D_4(a_1)}$ commute. Obviously
  $s_1s_2$ and $s_4s_{21}$ centralize $w_{A_2\widetilde{A}_2}$ and generate
  an elementary abelian subgroup of $C_W(w_{A_2\widetilde{A}_2})$ of order
  $9$. Then $C_W(w_{A_2\widetilde{A}_2})$ is generated by
  $\{s_1s_2,s_4s_{21}, w_{D_4(a_1)}\}$.  Define
  $\varphi_{A_2\widetilde{A}_2}$ by
  \[
  \varphi_{A_2\widetilde{A}_2 }(s_1s_2)= \zeta_3,\quad
  \varphi_{A_2\widetilde{A}_2}(s_4s_{21})= \zeta_3, \quad \text{and}\quad
  \varphi_{A_2\widetilde{A}_2}( w_{D_4(a_1)})=1.
  \]
  Then $\varphi_{A_2\widetilde{A}_2} (w_{A_2\widetilde{A}_2})= \zeta_3^2$.

  Using \cite[Theorem 4.2]{springer} it is easy to see that
  $C_W(w_{A_2\widetilde{A}_2})$ acts on its $\zeta_3$-eigenspace as the
  complex reflection group $G_5$. This group may be described by the diagram
  $\vcenter{\begin{xy}<.5cm,0cm>: (0,0)+<4pt,0pt>;(2,0) -<4pt,0pt>**\dir{-};
      ?*+!D{_4}; (0,0)*{_3}* \xycircle<4pt>{-}*++!U{_a};
      (2,0)*{_3}*\xycircle<4pt>{-}*++!U{_b} \end{xy}}$ where $a=s_1s_2$ and
  $b=s_2s_3s_2s_3s_4s_3$. Then $\varphi_{A_2\widetilde{A}_2}(a)=
  \varphi_{A_2\widetilde{A}_2} (b) =\zeta_3$. Also, $\varphi_{A_2
    \widetilde{A}_2} = \det|_{E(\zeta_{3})}$.

\item[$F_4(a_1)$] This class is regular and contains none of the elements
  of any proper reflection subgroup of $W$. The image of this class under
  multiplication by $w_0$ is the class labeled by
  $A_2\widetilde{A}_2$. Thus, $C_W\left( w_{F_4(a_1)} \right) =C_W \left(
    w_{A_2 \widetilde{A}_2} \right)$. We take $w_{F_4(a_1)}= w_{A_2
    \widetilde{A}_2}w_0$ and $\varphi_{F_4(a_1)} =\varphi_{A_2
    \widetilde{A}_2}$.  Then $w_{F_4(a_1)}$ has order six because
  $w_{A_2\widetilde{A}_2}$ has order three, so the $\zeta_6$-eigenspace of
  $w_{F_4(a_1)}$ is equal to the $\zeta_3$-eigenspace of $w_{A_2
    \widetilde{A}_2}$. Therefore, $C_W\left( w_{F_4(a_1)} \right)$ acts on
  its $\zeta_6$-eigenspace as the complex reflection group $G_5$ and
  $\varphi_{F_4(a_1)}= (\det|_{E(\zeta_{6})})^2$.

\item[$A_3\widetilde{A}_1$] This class contains the Coxeter elements in the
  reflection subgroups of $W$ of types $A_3\widetilde{A}_1$,
  $A_1\widetilde{A}_3$, $2A_1B_2$, and $2\widetilde{A}_1 B_2$ (see
  \cite{douglasspfeifferroehrle:reflection}). In addition, the image of this
  class under multiplication by $w_0$ is the class labeled $B_2$ containing
  $s_2s_3$. The reflections $s_{21}$ and $s_{24}$ corresponding to the
  highest short and long roots generate a subgroup of $W$ of type $B_2$.
  Taking $w_{A_3\widetilde{A}_1}=s_2s_3w_0$, we have
  \[
  C_W(w_{A_3\widetilde{A}_1}) = \left\langle w_{A_3
      \widetilde{A}_1}\right\rangle \times \left\langle s_{21},
    s_{24}\right\rangle \cong \cyclic{4} \times W(B_2).
  \] 
  Define $\varphi_{A_3\widetilde{A}_1}$ by $\varphi_{A_3 \widetilde{A}_1}
  (w_{A_3 \widetilde{A}_1}) =-1$ and $\varphi_{A_3
    \widetilde{A}_1}|_{W(B_2)} = \epsilon_{W(B_2)}$.
\end{description}

\bigskip With the same conventions as for $W(B_3)$, characters $\varphi_d$
satisfying the conclusion of Conjecture~B are summarized in the following
table.
\begin{multline*}
  \renewcommand{\arraystretch}{1.5}
  \begin{array}{c||c|cc|cc|cc|cc|cc|c|cc|c}
    d&4A_1& \multicolumn{2}{c|}{D_4}&  \multicolumn{2}{c|}{D_4(a_1)}
    &\multicolumn {2}{c|}{C_3A_1}& 
    \multicolumn{2}{c|}{A_2\widetilde{A}_2}&\multicolumn{2}{c|}{F_4(a_1)}&F_4&\multicolumn{2}{c|}{A_3\widetilde{A}_1}&B_4\\\hline 
    C_W(w_d)&W& \multicolumn {2}{c|} {\cyclic{6}\times \symmetric{3}}&\multicolumn{2}{c|}{G_8} &\multicolumn{2}{c|} {\cyclic{6} \times \symmetric{3}}&
    \multicolumn{2}{c|}{G_5} &\multicolumn{2}{c|}{G_5}&\cyclic{12} & \multicolumn{2}{c|}{\cyclic 4\times W(B_2)} &\cyclic{8}\\\hline 
    \text{Generators}&S&w_{D_4}&* &c&d& w_{C_3A_1}&*&
    a & b & a & b & w_{F_4} & w_{A_3\widetilde{A}_1}& * & w_{B_4}\\\hline    
    \varphi_d &\epsilon& \zeta_3&\epsilon&-1&-1&\zeta_3 & \epsilon &
    \zeta_3&\zeta_3& \zeta_3&\zeta_3 & \zeta_3 & -1 &\epsilon& \zeta_4     
  \end{array}
\end{multline*}

The values of the characters $\varphi_d^W$ together with $\rho_W^4$ and
$\omega_W^4$ are given in Table~\ref{fig:F4}.

\begin{table}[htbp]
\small
  \begin{multline*}
    \renewcommand{\arraystretch}{1.3}\arraycolsep4pt
    \begin{array}{r|rrrrrrrrrrrrr}
      &A_0&4A_1&2A_1&A_2&D_4&D_4(a_1)&\widetilde A_2&C_3A_1
      &A_2\widetilde{A}_2&F_4(a_1)&F_4&A_1&3A_1\\\hline
      \epsilon=\varphi_{4A_1}&1&1&1&1&1&1&1&1&1&1&1&-1&-1\\
      \varphi_{D_4}&32&32&\cdot&-1&-1&\cdot&2&2&-4&-4&\cdot&\cdot&\cdot\\
      \varphi_{D_4(a_1)}&12&12&4&\cdot&\cdot&8&\cdot&\cdot&6&6&2&\cdot
      &\cdot\\ 
      \varphi_{C_3A_1}&32&32&\cdot&2&2&\cdot&-1&-1&-4&-4&\cdot&-8&-8\\
      \varphi_{A_2\widetilde{A}_2}&16&16&\cdot&-2&-2&8&-2&-2&7&7&-1&
      \cdot&\cdot\\ 
      \varphi_{F_4(a_1)}&16&16&\cdot&-2&-2&8&-2&-2&7&7&-1&\cdot&\cdot\\
      \varphi_{F_4}&96&96&\cdot&\cdot&\cdot&16&\cdot&\cdot&-6&-6&-2&\cdot
      &\cdot\\
      \varphi_{A_3\widetilde{A}_1}&36&36&4&\cdot&\cdot&-12&\cdot&\cdot&\cdot
      &\cdot&\cdot&-6&-6\\ 
      \varphi_{B_4}&144&144&\cdot&\cdot&\cdot&-24&\cdot&\cdot&\cdot&\cdot
      &\cdot&\cdot&\cdot\\ \hline
      \rho_W^4&385&385&9&-2&-2&5&-2&-2&7&7&-1&-15&-15\\
      \omega_W^4&385&385&9&-2&-2&5&-2&-2&7&7&-1&15&15
    \end{array}\\ \\
    \renewcommand{\arraystretch}{1.3}\arraycolsep4pt
    \begin{array}{rrrrrrrrrrrr|l}
      A_1\widetilde{A}_2&C_{3}&A_3&\widetilde{A}_1
      &2A_1\widetilde{A}_1&A_2\widetilde{A}_1 &B_3 
      &B_2A_1&A_1\widetilde{A}_1&B_{2}&A_3\widetilde{A}_1&B_4\\\hline
      -1&-1&-1&-1&-1&-1&-1&-1&1&1&1&1&\varphi_{4A_1}=\epsilon\\
      \cdot&\cdot&\cdot&-8&-8&1&1&\cdot&\cdot&\cdot&\cdot&\cdot
      &\varphi_{D_4}\\ 
      \cdot&\cdot&\cdot&\cdot&\cdot&
      \cdot&\cdot&\cdot&\cdot&-4&-4&-2 &\varphi_{D_4(a_1)}\\ 
      1&1&\cdot&\cdot&\cdot&\cdot&\cdot&\cdot&\cdot&
      \cdot&\cdot&\cdot &\varphi_{C_3A_1}\\ 
      \cdot&\cdot&\cdot&\cdot&\cdot&\cdot&
      \cdot&\cdot&\cdot&\cdot&\cdot&\cdot &\varphi_{A_2\widetilde{A}_2}\\ 
      \cdot&\cdot&\cdot&\cdot&\cdot&\cdot&
      \cdot&\cdot&\cdot&\cdot&\cdot&\cdot &\varphi_{F_4(a_1)}\\ 
      \cdot&\cdot&\cdot&\cdot&\cdot&\cdot&\cdot&\cdot&
      \cdot&\cdot&\cdot&\cdot &\varphi_{F_4}\\ 
      \cdot&\cdot&2&-6&-6&\cdot&\cdot&2&
      \cdot&\cdot&\cdot&\cdot &\varphi_{A_3\widetilde{A}_1}\\ 
      \cdot&\cdot&\cdot&\cdot&\cdot&\cdot&\cdot&\cdot&
      \cdot&\cdot&\cdot&\cdot &\varphi_{B_4}\\ \hline
      \cdot&\cdot&1&-15&-15&\cdot&\cdot&1&1&-3&-3&-1 &\rho_W^4\\
      \cdot&\cdot&-1&15&15&\cdot&\cdot&-1&1&-3&-3&-1 &\omega_W^4
    \end{array}
  \end{multline*}
  \caption{The characters $\varphi_d^W$, $\rho_W^4$, and
    $\omega_W^4$ for $W(F_4)$}\label{fig:F4}
\end{table}

\subsection{\texorpdfstring{$W$}{W} of type \texorpdfstring{$H$}H} 

We label the conjugacy classes $C_1, C_2,\ldots$ and choose representatives
$w_1,w_2,\ldots$ as in \cite{geckpfeiffer:characters} and in \CHEVIE. For
each cuspidal class $C_i$ we write the character $\varphi_{w_i}$ as
$\varphi_i$.  When the index $i$ is fixed, we frequently denote $|w_i|$ by
$d$.

\subsubsection{\texorpdfstring{$W=W(H_3)$}{W=W(H3)}}
The cuspidal classes are $C_6$, $C_8$, $C_9$, and $C_{10}$. All these
classes are regular.

\medskip
\begin{description}
\item[$C_6$] $d=10$ and $C_W(w_6)=\langle w_6\rangle$. Define
  $\varphi_6(w_6)=\zeta_{10}$. Then $\varphi_6 = \det|_{E(\zeta_{10})}$.
 
\item[$C_8$] $d=6$ and $C_W(w_8)=\langle w_8\rangle$. Define
  $\varphi_8(w_8)=\zeta_{6}$. Then $\varphi_8 = \det|_{E(\zeta_{6})}$.

\item[$C_9$] $d=10$ and $C_W(w_9)=\langle w_9\rangle$. Define
  $\varphi_9(w_9)=\zeta_{10}$.  The elements $w_6$ and $w_9$ are related by
  $w_9=w_6^3$.  Thus the $\zeta_{10}^3$-eigenspace of $w_9$ coincides with
  the $\zeta_{10}$-eigenspace of $w_6$, and we have $\varphi_9 =
  (\det|_{E(\zeta_{10}^3)})^7$.

\item[$C_{10}$] The element $w_{10} = w_0$ is central and we take
  $\varphi_{10} =\epsilon$.
\end{description}
\bigskip 

The values of the characters $\varphi_i^W$ together with $\rho_W^3$ and
$\omega_W^3$ are given in Table~\ref{fig:H3}.

\begin{table}[htbp]
  \[
  \renewcommand{\arraystretch}{1.2}
  \begin{array}{r|rrrrrrrrrr}
    &C_1&C_2&C_3&C_4&C_5&C_6&C_7&C_8&C_9&C_{10}\\\hline
    \varphi_6&12&\cdot&\mu&\cdot&\cdot&-\nu&\nu&\cdot&-\mu&-12\\
    \varphi_8&20&\cdot&\cdot&\cdot&-1 &\cdot&\cdot&1&\cdot&-20\\
    \varphi_9&12&\cdot&\nu&\cdot&\cdot &-\mu&\mu&\cdot&-\nu&-12\\  
    \epsilon=\varphi_{10}&1&-1&1&1&1 &-1&1&-1&-1&-1\\ \hline
    \rho_W^3&45&-1&\cdot&1&\cdot &\cdot&\cdot&\cdot&\cdot&-45\\
    \omega_W^3&45&1&\cdot&1&\cdot  &\cdot&\cdot&\cdot&\cdot&45
  \end{array}
  \]
  \caption{Induced characters $W(H_3)$: $\mu=
    \zeta_5+\zeta_5^4$, $\nu=\zeta_5^2 +\zeta_5^3$}\label{fig:H3}
\end{table}

\subsubsection{\texorpdfstring{$W=W(H_4)$}{W=W(H4)}}
The cuspidal classes, the order of their elements, and the sizes of their
centralizers are listed in the next table.
\[
\begin{array}{ccc|ccc|ccc|ccc}
  w&\left|w\right|&\left|C_W\left(w\right)\right|
  &w&\left|w\right|&\left|C_W\left(w\right)\right|
  &w&\left|w\right|&\left|C_W\left(w\right)\right|
  &w&\left|w\right|&\left|C_W\left(w\right)\right|\\\hline
  w_{11}&30&30&w_{19}&10&50&w_{25}&6&36&w_{30}&10&600\\
  w_{14}&20&20&w_{21}&10&100&w_{26}&5&600&w_{31}&10&100\\
  w_{15}&15&30&w_{22}&15&30&w_{27}&5&50&w_{32}&3&360\\
  w_{17}&12&12&w_{23}&20&20&w_{28}&30&30&w_{33}&5&600\\
  w_{18}&10&600&w_{24}&6&360&w_{29}&4&240&w_{34}&2&14400\\
\end{array}
\]

For $i=11, 14, 17, 23, 28$ each of the elements $w_{i}$ is self-centralizing
and regular. We define $\varphi_i(w_i)= \zeta_d^2$ in all cases. Then,
$\varphi_i= (\det|_{E( \zeta_{d})})^2$ for $i=11, 14, 17$. For $i=23$ we
have $d= 20$, $E(\zeta_{20}^3)$ is a regular eigenspace, and $\varphi_{23}=
(\det|_{E( \zeta_{20}^3)})^{14}$.  For $i=28$ we have $d= 30$,
$E(\zeta_{30}^7)$ is a regular eigenspace, and $\varphi_{28}= (\det|_{E(
  \zeta_{30}^7)})^{26}$.

For $i=15, 22$ we have $d=15$ and $C_W(w_i)= \langle w_0w_i \rangle
\cong \cyclic{2} \times \langle w_i \rangle$. These classes are
regular. We define $\varphi_i(w_0w_i)= \zeta_{15}$ in both cases. For
$i=15$, $E(\zeta_{15})$ is a regular eigenspace. Notice that,
since $\zeta_{15} = (\zeta_{30}^{17})^{16}$, the
$\zeta_{15}$-eigenspace of $w_{15} = (w_0w_{15})^{16}$ is equal to
the $\zeta_{30}^{17}$-eigenspace
of $w_0w_{15}$. Thus $\varphi_{15}= (\det|_{E( \zeta_{15})})^{16}$. For
$i=22$, $E(\zeta_{15}^2)$ is a regular eigenspace and $\varphi_{22}=
(\det|_{E( \zeta_{15}^2)})^{8}$.

For $i=18, 26, 30, 33$ we have $C_W(w_i)= \langle w_{18},w_{19},w_{29}
\rangle$. These are regular classes and in all cases $C_W(w_i)$ acts on a
regular, two-dimensional eigenspace of $w_i$ as the complex reflection
group $G_{16}$. Define
\[
\varphi_{18}= \varphi_{33} \quad\text{by}\quad\left( w_{18}, w_{19},
  w_{29}\right) \mapsto\left( \zeta_5^2,\zeta_5^4,1\right)
\]
and
\[
\varphi_{26}= \varphi_{30} \quad \text{by}\quad\left( w_{18},w_{19},w_{29}
\right)\mapsto \left( \zeta_5^4, \zeta_5^3,1 \right).
\]
For $i=18$ we have $d=10$, $E(\zeta_{10})$ is a regular eigenspace, and
  $\varphi_{18}= (\det|_{E( \zeta_{10})})^2$.

For $i=26$ we have $d=5$, $E(\zeta_5)$ is a regular eigenspace, and
  $\varphi_{26}= (\det|_{E( \zeta_{5})})^4$.

For $i=30$ we have $d=10$, $E(\zeta_{10}^3)$ is a regular eigenspace,
  and $\varphi_{30}= (\det|_{E( \zeta_{10}^3)})^4$.

For $i=33$ we have $d=5$, $E(\zeta_5^2)$ is a regular eigenspace, and
  $\varphi_{33}= (\det|_{E( \zeta_{5}^2)})^2$.

For $i=19, 27$ we have $w_{27}= w_{19}^2$ and $C_W(w_i)= \left\langle
  w_{18}\right\rangle \times \left\langle w_{27}\right\rangle \cong
\cyclic{10} \times \cyclic{5}$. Define
\[
\varphi_{19} =\varphi_{27} \quad \text{by} \quad \left(w_{18}, w_{27}
\right) \mapsto\left( \zeta_5^3,\zeta_5^4\right).
\]

For $i=21, 31$ we have $w_{31}= w_{21}^3$ and $C_W(w_i)= \left\langle
w_{18},w_{27},s_2\right\rangle$.
Define
\[
\varphi_{21} \quad \text{by} \quad \left(w_{18},w_{27}, s_2 \right)
\mapsto\left( \zeta_5, \zeta_5^2, -1\right) \quad \text{and} \quad
  \varphi_{31}= \varphi_{21}^3.
\]

For $i=24, 32$ we have $C_W(w_i)= \langle w_{24},w_{25},w_{29}
\rangle$. Denote this group simply by $Z$. The classes $C_{24}$ and $C_{32}$
are regular. The representatives $w_{24}$ and $w_{32}$ have order $d=6$ and
$d=3$ respectively and are related by $w_{32}= w_{24}^2$.  Thus, the
$\zeta_6$-eigenspace of $w_{24}$ is equal to the $\zeta_3$-eigenspace of
$w_{32}$. Denote this vector space simply by $E$. Then $E$ is a regular,
two-dimensional eigenspace for $w_{24}$ and $w_{32}$, and $Z$ acts on $E$ as
the complex reflection group $G_{20}$. Define
\[
\varphi_{24}= \varphi_{32} \quad \text{by}\quad \left( w_{24},w_{25},w_{29}
\right) \mapsto \left( \zeta_3^2, \zeta_3,1 \right).
\]
In both cases we have $\varphi_{i}= (\det|_{E( \zeta_{6})})^2$.

For $i=25$ we have $C_W(w_{25})= \left\langle w_0w_{24},w_{25},s_2\right\rangle$.
Define
\[
\varphi_{25} \quad \text{by} \quad \left( w_0w_{24}, w_{25},s_2 \right)
\mapsto \left( \zeta_3^2, \zeta_3^2, -1 \right).
\]

For $i=29$ we have $C_W(w_{29})= \langle w_{18},w_{24},w_{29} \rangle$. This
class is regular. We have $d=4$, $E(\zeta_4)$ is a regular, two-dimensional
eigenspace, and $C_W(w_{29})$ acts on $E(\zeta_4)$ as the complex reflection
group $G_{22}$.  Define
\[
\varphi_{29} \quad \text{by}\quad \left(w_{18},w_{24},w_{29}\right) \mapsto
\left( 1,1,-1 \right).
\]
Then $\varphi_{29}= \det|_{E( \zeta_{4})}$.

Finally, for $i=34$ we have $w_{34}=w_0$ and we define $\varphi_{34}
=\epsilon$.
 
The values of the characters $\varphi_i^W$ together with $\rho_W^4$ and
$\omega_W^4$ are given in Table~\ref{fig:H4}.

\begin{table}[htbp]
  \tiny
  \begin{multline*}
  \renewcommand{\arraystretch}{1.2}\arraycolsep4pt
  \begin{array}{r|rrrrrrrrrrrrrrrrr}
    &C_{1}&C_{2}&C_{3}&C_{4}&C_{5}&C_{6} &C_{7}&C_{8}&C_{9}&C_{10}&C_{11}&C_{12}&C_{13}&C_{14} &C_{15}&C_{16}&C_{17}\\\hline 
    \varphi_{11}&480&\cdot&\cdot&\cdot&\cdot &\cdot&\cdot&\cdot&\cdot&\cdot&-\nu&\cdot&\cdot&\cdot &-\mu&\cdot&\cdot\\  
    \varphi_{14}&720&\cdot&\cdot&\cdot&\cdot&\cdot&\cdot&\cdot&\cdot&\cdot&\cdot&\cdot&\cdot&-2\nu&\cdot&\cdot&\cdot\\ 
    \varphi_{15}&480&\cdot&\cdot&\cdot& \cdot&\cdot&\cdot&\cdot&\cdot&\cdot&-\mu&\cdot &\cdot&\cdot&-\nu&\cdot&\cdot\\
    \varphi_{17}&1200&\cdot&\cdot&\cdot&\cdot&\cdot&\cdot&\cdot&\cdot&\cdot&\cdot& \cdot&\cdot&\cdot&\cdot &\cdot&2\\ 
    \varphi_{18}&24&\cdot&2\nu&\cdot&\cdot&\cdot&\cdot&\cdot&\cdot&2\mu&\mu&\cdot&\cdot&\mu&\nu &\cdot&\cdot\\
    \varphi_{19}&288&\cdot&-2&\cdot&\cdot&\cdot&\cdot&\cdot&\cdot&-2&\cdot&\cdot&\cdot&\cdot&\cdot&\cdot&\cdot\\
    \varphi_{21}&144&-12&1{-}\mu&\cdot&\cdot&-\mu&-\nu&\cdot&\cdot&1{-}\nu&\cdot&\cdot&-\mu&\cdot&\cdot&-\nu&\cdot\\ 
    \varphi_{22}&480&\cdot&\cdot&\cdot&\cdot&\cdot&\cdot&\cdot&\cdot&\cdot&-\nu&\cdot&\cdot&\cdot&-\mu&\cdot&\cdot\\ 
    \varphi_{23}&720&\cdot&\cdot& \cdot&\cdot&\cdot&\cdot&\cdot&\cdot&\cdot&\cdot&\cdot&\cdot&-2\mu&\cdot &\cdot&\cdot\\  
    \varphi_{24}&40&\cdot&\cdot&\cdot&-2&\cdot&\cdot&\cdot&\cdot&\cdot&-1&\cdot&\cdot&\cdot&-1&\cdot&-1\\
    \varphi_{25}&400&-20&\cdot&\cdot&1&\cdot&\cdot&1&\cdot&\cdot&\cdot&1&\cdot&\cdot&\cdot&\cdot&\cdot\\
    \varphi_{26}&24&\cdot&2\mu&\cdot&\cdot&\cdot&\cdot&\cdot&\cdot&2\nu&\nu&\cdot&\cdot&\nu&\mu&\cdot&\cdot\\ 
    \varphi_{27}&288&\cdot&-2&\cdot&\cdot&\cdot&\cdot&\cdot&\cdot&-2&\cdot&\cdot&\cdot&\cdot&\cdot&\cdot&\cdot\\
    \varphi_{28}&480&\cdot&\cdot&\cdot&\cdot&\cdot&\cdot&\cdot&\cdot&\cdot&-\mu&\cdot&\cdot&\cdot&-\nu&\cdot&\cdot\\ 
    \varphi_{29}&60&\cdot&\cdot&-4&\cdot&\cdot&\cdot&\cdot&\cdot&\cdot&\cdot&\cdot&\cdot&-2&\cdot&\cdot&-2\\
    \varphi_{30}&24&\cdot&2\mu&\cdot&\cdot&\cdot&\cdot&\cdot&\cdot&2\nu&\nu&\cdot&\cdot&\nu&\mu&\cdot&\cdot\\
    \varphi_{31}&144&-12&1{-}\nu&\cdot&\cdot&-\nu&-\mu&\cdot&\cdot&1{-}\mu&\cdot&\cdot&-\nu&\cdot&\cdot&-\mu&\cdot\\ 
    \varphi_{32}&40&\cdot&\cdot&\cdot&-2&\cdot&\cdot&\cdot&\cdot&\cdot&-1&\cdot&\cdot&\cdot&-1&\cdot&-1\\
    \varphi_{33}&24&\cdot&2\nu&\cdot&\cdot&\cdot&\cdot&\cdot&\cdot&2\mu&\mu&\cdot&\cdot&\mu&\nu&\cdot&\cdot\\
    \epsilon=\varphi_{34}&1&-1&1&1&1&-1&-1&-1&-1&1&1&-1&-1&1&1&-1&1\\ \hline
    \rho_W^4&6061&-45&-4&-3&-2&\cdot&\cdot&\cdot&-1&-4&-1&\cdot&\cdot&-1&-1&\cdot&-1\\
    \omega_W^4&6061&45&-4&-3&-2&\cdot&\cdot&\cdot&1&-4&-1&\cdot&\cdot&-1&-1&\cdot&-1\\
  \end{array}\\
    \renewcommand{\arraystretch}{1.2}\arraycolsep4pt
    \begin{array}{rrrrrrrrrrrrrrrrr|l}
      C_{18}&C_{19}&C_{20}&C_{21}&C_{22}&C_{23}&C_{24}&C_{25}&C_{26}&C_{27}&C_{28}&C_{29}&C_{30}&C_{31}&C_{32}&C_{33}&C_{34}&\\\hline
      20\mu&\cdot&\cdot&\cdot&-\nu&\cdot&-12&\cdot&20\nu&\cdot&-\mu&\cdot&20\nu&\cdot&-12&20\mu&480&\varphi_{11}\\ 
      30\mu&\cdot&\cdot&\cdot&\cdot&-2\mu&\cdot&\cdot&30\nu&\cdot&\cdot&-24&30\nu&\cdot&\cdot&30\mu&720&\varphi_{14}\\ 
      20\nu&\cdot&\cdot&\cdot&-\mu&\cdot&-12&\cdot&20\mu&\cdot&-\nu&\cdot&20\mu&\cdot&-12&20\nu&480&\varphi_{15}\\ 
      \cdot&\cdot &\cdot&\cdot&\cdot&\cdot&-30&\cdot&\cdot&\cdot&\cdot&-40&\cdot&\cdot&-30&\cdot&1200&\varphi_{17}\\ 
      11{-}\mu&-1&\cdot&2\mu&\mu&\nu&12&\cdot&11{-}\nu&-1&\nu&12&11{-}\nu&2\nu&12&11{-}\mu&24&\varphi_{18}\\ 
      -12&3&\cdot&-2&\cdot&\cdot&\cdot&\cdot&-12&3&\cdot&\cdot&-12&-2&\cdot&-12&288&\varphi_{19}\\
      12\mu&-1&-12&1{-}\nu&\cdot&\cdot&\cdot&\cdot&12\nu&-1&\cdot&\cdot&12\nu&1{-}\mu&\cdot&12\mu&144&\varphi_{21}\\ 
      20\mu&\cdot&\cdot&\cdot&-\nu&\cdot&-12&\cdot&20\nu&\cdot&-\mu&\cdot&20\nu&\cdot&-12&20\mu&480&\varphi_{22}\\ 
      30\nu&\cdot&\cdot&\cdot&\cdot&-2\nu&\cdot&\cdot&30\mu&\cdot&\cdot&-24&30\mu&\cdot&\cdot&30\nu&720&\varphi_{23}\\ 
      20&\cdot&\cdot&\cdot&-1&\cdot&19&-2&20&\cdot&-1&20&20&\cdot&19&20&40&\varphi_{24}\\
      \cdot&\cdot&-20&\cdot&\cdot&\cdot&-20&1&\cdot&\cdot&\cdot&\cdot&\cdot&\cdot&-20&\cdot&400&\varphi_{25}\\
      11{-}\nu&-1&\cdot&2\nu&\nu&\mu&12&\cdot&11{-}\mu&-1&\mu&12&11{-}\mu&2\mu&12&11{-}\nu&24&\varphi_{26}\\ 
      -12&3&\cdot&-2 &\cdot&\cdot&\cdot&\cdot&-12&3&\cdot&\cdot&-12&-2&\cdot&-12&288&\varphi_{27}\\
      20\nu&\cdot&\cdot&\cdot&-\mu&\cdot&-12&\cdot&20\mu&\cdot&-\nu&\cdot&20\mu&\cdot&-12&20\nu&480&\varphi_{28}\\ 
      30&\cdot&\cdot&\cdot&\cdot&-2&30&\cdot&30&\cdot&\cdot&28&30&\cdot&30&30&60&\varphi_{29}\\
      11{-}\nu&-1&\cdot&2\nu&\nu&\mu&12&\cdot&11{-}\mu&-1&\mu&12&11{-}\mu&2\mu&12&11{-}\nu&24&\varphi_{30}\\ 
      12\nu&-1&-12&1{-}\mu&\cdot&\cdot&\cdot&\cdot&12\mu&-1&\cdot&\cdot&12\mu&1{-}\nu&\cdot&12\nu&144&\varphi_{31}\\ 
      20&\cdot&\cdot&\cdot&-1&\cdot&19&-2&20&\cdot&-1&20&20&\cdot&19&20&40&\varphi_{32}\\
      11{-}\mu&-1&\cdot&2\mu&\mu&\nu&12&\cdot&11{-}\nu&-1&\nu&12&11{-}\nu&2\nu&12&11{-}\mu&24&\varphi_{33}\\ 
      1&1&-1&1&1&1&1&1&1&1&1&1&1&1&1&1&1&\varphi_{34}\\ \hline
      11&1&-45&-4&-1&-1&19&-2&11&1&-1&29&11&-4&19&11&6061&\rho_W^4\\
      11&1&45&-4&-1&-1&19&-2&11&1&-1&29&11&-4&19&11&6061&\omega_W^4\\
    \end{array}
  \end{multline*}
    \caption{Induced characters for $W\left(H_4\right)$:
    $\mu=\zeta_5+\zeta_5^4$, $\nu=\zeta_5^2 +\zeta_5^3$}\label{fig:H4}
\end{table}

\appendix
\section{Bulky Parabolic Subgroups}\label{a:bulky}
For each finite irreducible Coxeter group $W$ the following table lists
the types of all bulky parabolic subgroups of $W$ other than $W$ itself,
the trivial subgroup, and the subgroup of type $A_1$. This information
has been extracted from the results in \cite{howlett:normalizers}.

\[
\renewcommand{\arraystretch}{1.3}
\begin{array}{c|l}
  W&\text{Bulky Parabolic Subgroups}\\\hline
  A_n&\text{$A_{n_1}A_{n_2}\cdots A_{n_k}$ with
  $n_i$ distinct and $\sum_{i=1}^kn_i\le n+k-1$}\\
  B_n&
  \text{$B_j$ with $1\le j\le n-1$},\quad 
  \text{$A_1B_j$ with $1\le j\le n-2$}\\ 
  D_n,\ \text{$n$ even}&A_1 D_{n-2}\\
  D_n,\ \text{$n$ odd}&A_1 D_{n-2},\quad A_1 A_{n-3},\quad A_{n-1}\\ 
  E_6&A_1 A_2,\quad A_1 A_3,\quad A_4,\quad A_1 A_4,\quad A_5,\quad D_5\\ 
  E_7&D_6\\
  E_8&E_7\\
  F_4&\widetilde{A}_1,\quad A_1\widetilde{A}_1,\quad B_2,\quad
  B_3,\quad C_3\\ 
  H_3&A_1^2\\
  H_4&H_3\\
  I_2(m),\ \text{$m$ even}&\widetilde{A}_1
\end{array}
\]

\textbf{Acknowledgement.} We acknowledge
support from the DFG-priority program SPP1489 ``Algorithmic
and Experimental Methods in Algebra, Geometry, and Number Theory''.  The
third author also wishes to thank Science Foundation Ireland for its support.


\end{document}